\documentclass{llncs}
\usepackage{amssymb, amsmath}
\usepackage{graphicx,color}
\usepackage{algorithmic}
\usepackage[section]{algorithm}
\newcommand{\bbox}{\rule{0.6em}{0.6em}}
\newcommand{\boxi}{\ensuremath{box}}
\newcommand{\cubi}{\ensuremath{cub}}
\begin{document}
\date{}
\pagestyle{plain}
\title{Cubicity, Degeneracy, and Crossing Number}
% \author{L. Sunil Chandran \thanks{Department of Computer Science \& Automation, 
% Indian Institute of Science, Bangalore, India. email:sunil@csa.iisc.ernet.in} \and 
% Rogers Mathew\thanks{Department of Computer Science \& Automation, 
% Indian Institute of Science, Bangalore, India. email:rogers@csa.iisc.ernet.in} \and 
% Naveen Sivadasan\thanks{Department of Computer Science \& Engineering, 
% Indian Institute of Technology, Hyderabad, India. email:nsivadasan@iith.ac.in}}
\author{Abhijin Adiga \and L. Sunil Chandran \and Rogers Mathew}
% \institute{}
\institute{Department of Computer Science and Automation, \\ Indian Institute of
Science, \\
Bangalore - 560012, India. \\ \email{\{abhijin,sunil,rogers\}@csa.iisc.ernet.in}} 
% \and 
% Department of Computer Science and Engineering, \\
% Indian Institute of Technology, \\
% Hyderabad - 502205, India. \\ \email{nsivadasan@iith.ac.in}}
% \texttt{\{sunil,rogers\}@csa.iisc.ernet.in}}
\maketitle
\bibliographystyle{plain}
\begin{abstract}
A $k$-box $B=(R_1,\ldots,R_k)$, where each $R_i$ is a closed interval on
the real line, is defined to be the Cartesian product $R_1\times R_2\times
\cdots\times R_k$. If each $R_i$ is a unit length interval, we call $B$ a $k$-cube. 
\emph{Boxicity} of a graph $G$, denoted as $\boxi(G)$, is the minimum integer $k$
such that $G$ is an intersection graph of $k$-boxes. Similarly, the \emph{cubicity}
of $G$, denoted as $\cubi(G)$, is the minimum integer $k$ such that $G$ is
an intersection graph of $k$-cubes. 

It was shown in [L. Sunil Chandran, Mathew C. Francis, and Naveen Sivadasan:  
Representing graphs as the intersection of axis-parallel cubes. {\em MCDES-2008, IISc Centenary Conference}, available at {\em CoRR}, abs/cs/\\0607092, 2006.] that, for a graph $G$ with maximum degree $\Delta$, 
$\cubi(G)\leq \lceil 4(\Delta +1)\log n\rceil$. In this paper, we show that, for a $k$-degenerate 
graph $G$, $\cubi(G) \leq (k+2)  \lceil 2e \log n \rceil$. Since $k$ is at most 
$\Delta$ and can be much lower, this clearly is a stronger result. This bound is tight.  We   
also give an efficient deterministic algorithm that runs in 
$O(n^2k)$ time to output a 
$8k(\lceil 2.42 \log n\rceil + 1)$ dimensional cube representation for $G$. 

The \emph{crossing number} of a graph $G$, denoted as $CR(G)$, 
is the minimum number of crossing pairs of edges, over all drawings of $G$ in the plane. 
An important consequence of the above result is that 
if the crossing number of a graph $G$ is $t$, then 
$\boxi(G)$ is $O(t^{\frac{1}{4}}{\lceil\log t\rceil}^{\frac{3}{4}})$ . This bound is 
tight up to a factor of $O((\log t)^{\frac{1}{4}})$.  We also show that, if $G$ has $n$ vertices, then $\cubi(G)$ is $O(\log n + t^{1/4}\log t)$. 

Let $(\mathcal{P},\leq)$ be a partially ordered set and let $G_{\mathcal{P}}$ denote 
its underlying comparability graph. Let $dim(\mathcal{P})$ denote the \emph{poset 
dimension} of $\mathcal{P}$. 
Another interesting consequence of our result is to show that 
$dim(\mathcal{P}) \leq 2(k+2)  \lceil 2e \log n \rceil$, where $k$ denotes the 
degeneracy of $G_{\mathcal{P}}$. 
Also, we get a deterministic algorithm that runs in $O(n^2k)$ time to construct a 
$16k(\lceil 2.42 \log n\rceil + 1)$ sized realizer for $\mathcal{P}$. 
As far as we know, though very good upper bounds exist for poset dimension in terms 
of maximum degree of its underlying comparability graph, no upper bounds in terms of 
the degeneracy of the underlying comparability graph is seen in the literature. 

It was shown in [L. Sunil Chandran, Mathew C. Francis, and Naveen Sivadasan: Geometric Representation of Graphs in Low Dimension Using 	Axis Parallel Boxes. Algorithmica 56(2): 129-140, 2010.] that boxicity of almost all graphs in $\mathcal{G}(n,m)$ model is $O(d_{av}\log n)$, where $d_{av} = \frac{2m}{n}$ denotes the average degree of the graph under consideration. 
In this paper, we prove a stronger result. Using our bound for the cubicity of $k$-degenerate graphs, we show that cubicity of almost all graphs in $\mathcal{G}(n,m)$ model is $O(d_{av}\log n)$. 
\end{abstract}

\section{Introduction}
A graph $G$ is an \emph{intersection graph} of sets from a family of sets
$\mathcal{F}$, if there exists $f:V(G)\rightarrow \mathcal{F}$ such that
$(u,v)\in E(G)\Leftrightarrow f(u)\cap f(v)\not=\emptyset$.
Representations of graphs as the intersection graphs of various geometrical
objects is a well studied topic in graph theory. Probably the most well studied
class of intersection graphs are the \emph{interval graphs}. 
Interval graphs are the intersection graphs of closed intervals on the real line. 
A restricted form of interval graphs,
that allow only intervals of unit length, are \emph{indifference graphs} or 
\emph{unit interval graphs}.

An interval
on the real line can be generalized to a ``$k$-box'' in $\mathbb{R}^k$.
A $k$-box $B=(R_1,\ldots,R_k)$, where each $R_i$ is a closed interval on
the real line, is defined to be the Cartesian product $R_1\times R_2\times
\cdots\times R_k$.
If each $R_i$ is a unit length interval, we call $B$ a $k$-cube.
Thus, 1-boxes are just closed intervals on the real line whereas 2-boxes
are axis-parallel rectangles in the plane. 
The parameter
boxicity of a graph $G$, denoted as $\boxi(G)$, is the minimum integer $k$
such that $G$ is an intersection graph of $k$-boxes. Similarly, the cubicity
of $G$, denoted as $\cubi(G)$, is the minimum integer $k$ such that $G$ is
an intersection graph of $k$-cubes. Thus, interval graphs are the graphs
with boxicity equal to 1 and unit interval graphs are the graphs with
cubicity equal to 1. 
A \emph{$k$-box representation} or a \emph{$k$ dimensional box representation} 
of a graph $G$ is a mapping of the vertices of $G$ to
$k$-boxes such that  two vertices in $G$ are adjacent if and only if their
corresponding $k$-boxes have a non-empty intersection. In a similar way, 
we define \emph{$k$-cube representation} (or \emph{$k$ dimensional cube representation} ) of a graph $G$.  
Since $k$-cubes by definition are also 
$k$-boxes, boxicity of a graph is at most its cubicity. 

The concepts of boxicity and cubicity were introduced by F.S. Roberts in 1969 \cite{Roberts}. 
Roberts showed that for any graph $G$ on $n$ vertices
$\boxi(G)\leq \lfloor \frac{n}{2}\rfloor $ and $\cubi(G)\leq \lfloor \frac{2n}{3} \rfloor$. 
Both these bounds are tight since
$\boxi(K_{2,2,\ldots,2})=\lfloor \frac{n}{2}\rfloor$ and $\cubi(K_{3,3,\ldots,3})=
\lfloor \frac{2n}{3} \rfloor$ where
$K_{2,2,\ldots,2}$ denotes the complete $n/2$-partite graph with 2 vertices
in each part and $K_{3,3,\ldots,3}$ denotes the complete $n/3$-partite graph
with 3 vertices in each part. 
It is easy to see that the boxicity of any graph is at least the boxicity
of any induced subgraph of it. 

Box representation of graphs finds application in niche overlap (competition) in
ecology and to problems of fleet maintenance in operations research
(see \cite{CozRob}). Given a low dimensional box representation,  
some well known NP-hard problems become polynomial time solvable. 
For instance, the max-clique problem is polynomial time solvable for graphs with 
boxicity $k$ because the number of maximal cliques in such graphs is only 
$O((2n)^k)$.  
\subsection{Previous Results on Boxicity and Cubicity}
\label{secboxresults}
It was shown by Cozzens \cite{Coz} that computing the boxicity of a graph
is \textbf{NP}-hard. 
% This was later improved by Yannakakis \cite{Yan1}, and finally by 
Kratochv\'{\i}l \cite{Kratochvil} 
showed that deciding whether the boxicity of a graph is at most 2 itself 
is \textbf{NP}-complete. It has been shown by Yannakakis \cite{Yan1} that deciding whether the 
cubicity of a given graph is at least 3 is \textbf{NP}-hard. 

Researchers have tried to bound the boxicity and cubicity of 
graph classes  with special structure. Scheinerman \cite {Scheiner} 
showed that the boxicity of outerplanar graphs is at most $2$.
Thomassen \cite {Thoma1} proved that the boxicity of planar graphs is
bounded from above by $3$. Upper bounds for the boxicity of many other graph
classes such as chordal graphs, AT-free graphs, permutation graphs etc.
were shown in \cite{CN05} by relating the boxicity of a graph with its
treewidth. The cube representation of special classes of graphs like hypercubes and
complete multipartite graphs were investigated in \cite{Roberts,Maehara,Quint}. 
% Also, the cubicity of the $d$-dimensional hypercube was shown to be $\Theta(\frac{d}{\log d})$ in
% \cite{CubHyp}. A lower bound for the cubicity of general graphs in terms of
% the diameter and maximum independent set size was shown in \cite{CMO}.

Various other upper bounds on
boxicity and cubicity in terms of graph parameters such as maximum degree, treewidth etc. 
can be seen in \cite{tech-rep,CFNfull06,CFNMaxdeg,Esperet,CN05}. 
The ratio of cubicity to boxicity of any graph on $n$ vertices was shown
to be at most $\lceil \log_2 n\rceil$ in \cite{CubBox}. 
% Chandran and Adiga \cite{AdigSun} showed that, for an interval graph $G$, 
%  $\lceil \log_2 \psi(G) \rceil \leq \cubi(G) \leq \lceil \log_2 \psi(G) \rceil + 2$, 
% where $\psi(G)$ is the claw number of $G$. 

% Various other upper bounds on
% boxicity and cubicity in terms of graph parameters such as maximum degree and treewidth
% were proved by Chandran, Francis and Sivadasan. In \cite{tech-rep}  they showed
% that, for any graph $G$ on $n$ vertices having maximum degree $\Delta$,
% $\boxi(G) \leq (\Delta + 2)\ln n $. 
% It was shown in \cite{CFNfull06} that $\cubi(G)\leq \lceil 4(\Delta +1)\ln n\rceil$. 
% They also found an upper bound for boxicity solely in terms of the maximum
% degree $\Delta$
% of a graph by showing that $\boxi(G) \leq 2\Delta^2$ \cite{CFNMaxdeg}. 
% This means that the boxicity of bounded degree graphs is bounded no
% matter what the size of the vertex set is. 
% Esperet \cite{Esperet} improved this bound to $\Delta^2 + 2$. 
% It was shown in \cite{CN05} by Chandran and Sivadasan that $\boxi(G) \leq
% \mathrm{tw}(G)+ 2$, where $\mathrm{tw}(G)$ denotes the treewidth of graph $G$. 
% Chandran and Adiga \cite{AdigSun} showed that, for an interval graph $G$, 
%  $\lceil \log_2 \psi(G) \rceil \leq \cubi(G) \leq \lceil \log_2 \psi(G) \rceil + 2$, 
% where $\psi(G)$ is the claw number of $G$. 

\subsection{Equivalent Definitions for Boxicity and Cubicity}
Let $G$ and $G_1,...,G_b$ be graphs such that $V(G_i)=V(G)$ for $1\le i\le b$.
% Let $G,G_1,G_2,\ldots,G_b$ be a collection of graphs with $V(G) = V(G_i)$,  
% for every $i \leq b$. 
We say $G= \bigcap_{i=1}^{b}G_i$ 
when $E(G) = \bigcap_{i=1}^{b}E(G_i)$. 
Below, we state two very useful lemmas due to Roberts \cite{Roberts}.
\begin{lemma}
\label{robertslem}
For any graph $G$, $\boxi(G)\leq k$ if and only if there exist $k$ interval
graphs $I_1,\ldots,I_k$ such that $G=I_1\cap \cdots\cap I_k$.
\end{lemma}
% Similar to boxicity, cubicity too has an equivalent definition due to Roberts. 
\begin{lemma}
\label{lemrobertscub}
For any graph $G$, $\cubi(G)\leq k$ if and only if there exist $k$ indifference
graphs (unit interval graphs) $I_1,\ldots,I_k$ such that $G=I_1\cap\cdots
\cap I_k$.
\end{lemma}

\subsection{Our Results}
A graph $G$ is \emph{$k$-degenerate} if the vertices of $G$ can be 
enumerated in such a way that every vertex is succeeded by at most $k$ of its neighbors. 
The least number $k$ such that $G$ is $k$-degenerate is called the degeneracy of $G$ and 
any such enumeration is referred to as a \emph{degeneracy order} of $V(G)$. For example, 
trees and forests are 1-degenerate and planar graphs are 5-degenerate. 
Series-parallel graphs, outerplanar graphs, non-regular cubic graphs, 
circle graphs of girth at least 5 etc. are subclasses of 2-degenerate graphs. 

\textbf{Main Result: }It was shown in \cite{CFNfull06} that, for a graph $G$ with maximum degree $\Delta$, 
$\cubi(G)\leq \lceil 4(\Delta +1)\log n\rceil$. In this paper, we show that, for a $k$-degenerate 
graph $G$, $\cubi(G) \leq (k+2)  \lceil 2e \log n \rceil$. Since $k$ is at most 
$\Delta$ and can be much lower, this clearly is a stronger result. We prove that this bound is tight. Moreover, 
we give an \emph{efficient deterministic algorithm} that outputs a 
$8k(\lceil 2.42 \log n\rceil + 1)$ dimensional cube representation for $G$ in 
$O(n^2k)$ time. 
% Any algorithm that produces an $O(k\log n)$ 
% dimensional cube representation for $G$ requires $O(nk \log n)$ 
% time to output the representation. Therefore, our algorithm is optimal 
% when $k \leq c\log n$, where $c$ is a constant. 

\textbf{Consequence 1: }The \emph{crossing number} of a graph $G$, denoted as $CR(G)$, 
is the minimum number of crossing pairs of edges, over all drawings of $G$ in the plane. 
We prove that, if $CR(G) = t$, then 
$\boxi(G) \leq 66t^{\frac{1}{4}}{\lceil\log 4t\rceil}^{\frac{3}{4}} + 6$. This bound is 
tight up to a factor of $O((\log t)^{\frac{1}{4}})$. We also show that, if $G$ has $n$ vertices, then $\cubi(G)$ is $O(\log n + t^{1/4}\log t)$. See Section \ref{CrossingSection} for details.  

\textbf{Consequence 2: }It was shown in \cite{tech-rep} that boxicity of almost all graphs in $\mathcal{G}(n,m)$ model is $O(d_{av}\log n)$, where $d_{av} = \frac{2m}{n}$ denotes the average degree of the graph under consideration. What can we infer about the cubicity of almost all graphs from the result of \cite{tech-rep}? It was shown in \cite{CubBox} that for every graph $G$, $\cubi(G) \leq \log_2n \times \boxi(G)$. Combining this result with that of \cite{tech-rep}, we can infer that cubicity of almost all graphs is $O(d_{av} \log^2n)$. 
In this paper, we prove a stronger result. Using our bound for the cubicity of $k$-degenerate graphs, we show that cubicity of almost all graphs in $\mathcal{G}(n,m)$ model is $O(d_{av}\log n)$. See Section \ref{CubicityRandGraph} for details. 

\textbf{Consequence 3: }Let $(\mathcal{P}, \leq)$ be a poset (partially ordered set) and 
let $G_{\mathcal{P}}$ be the underlying comparability graph of $\mathcal{P}$. A linear extension 
$L$ of $\mathcal{P}$ is a total order which satisfies $(x \leq y \in \mathcal{P}) \implies 
(x \leq y \in L)$. A realizer of $\mathcal{P}$ is a set of linear extensions of $\mathcal{P}$, 
say $\mathcal{R}$, which satisfy the following condition: for any two distinct elements $x$ and $y$, 
$x\leq y$ in $\mathcal{P}$ if and only if $x \leq y$ in $L$, $\forall L \in \mathcal{R}$.  
The \emph{poset dimension} of $\mathcal{P}$, denoted by $dim(\mathcal{P})$, 
is the minimum integer $k$ such that there exists a realizer of $\mathcal{P}$ of cardinality $k$. 
Yannakakis \cite{Yan1} showed that it is NP-complete to decide whether the dimension of a poset is at most 3.  
The poset dimension is an extensively studied parameter in the theory of 
partial order (See \cite{trotter2001combinatorics} for a comprehensive treatment).

There are several research papers in the partial order 
literature which study the dimension of posets whose underlying 
comparability graph has some special structure -- 
interval order, semi order and crown posets are some examples. 
While very good upper bounds (for example $c\Delta (\log \Delta)^2$, where $c$ is a constant) are known 
for poset dimension in terms of maximum degree $\Delta$ of its underlying comparability graph, 
as far as we know there are no upper bounds in terms of the degeneracy of the 
underlying comparability graph. Connecting our main result with a result in 
\cite{DiptAdiga}, we can get an upper bound for poset dimension in terms of the 
degeneracy of the underlying  comparability graph as follows. 
It was shown in \cite{DiptAdiga} that $dim(\mathcal{P}) < 2\boxi(G_{\mathcal{P}})$. Therefore, 
if the degeneracy of the underlying comparability graph $G_{\mathcal{P}}$ is $k$, 
then our result says that $dim(\mathcal{P}) \leq 2(k+2)  \lceil 2e \log n \rceil$. Also, 
we get a deterministic algorithm that runs in $O(n^2k)$ time to construct a 
$16k(\lceil 2.42 \log n\rceil + 1)$ sized realizer for $\mathcal{P}$. 

% \textbf{Consequence 3: }Let $S=\{\sigma_1, \sigma_2, \ldots , \sigma_p\}$ be a set of permutations of 
% $[n]$, where $n$ is any finite positive integer. 
% $S$ is called \emph{$k$-suitable} for $[n]$ if for any $k$-element subset 
% $X \subseteq [n]$ and for any $x \in X$, there exists a permutation $\sigma \in S$ with the 
% following property: 
% \[\sigma^{-1}(x) \geq \sigma^{-1}(y), \forall y \in X.\] 
% The minimum cardinality of a $k$-suitable set for $[n]$ is 
% denoted by $N(n,k)$. Note that $N(n,2) \leq 2$. Spencer \cite{scramble} proved that 
% $N(n,k) \leq k2^k\log_2\log_2 n$. F\"uredi and Kahn in \cite{FurediKahn} 
% gave a simple proof using Lovasz local lemma to show that 
% $N(n,k) \leq k^2(1 + \log \frac{n}{k})$. Both the proofs are 
% non-constructive. From our main deterministic algorithm, we can also get 
% an $O(n^2k)$ time algorithm to construct 
% a $k$-suitable set of cardinality at most $16(k-1)^2\lceil 2.42 \log n\rceil$, when 
% $k >2$ and $n>1$. 
% See Section \ref{DegandSuitability} for details. 

\section{Preliminaries}
For any finite positive integer $n$, let $[n]$ denote the 
set $\{1,\ldots n\}$. Unless mentioned explicitly, all logarithms 
are to the base $e$ in this paper. 
All the graphs that we consider are simple, finite and undirected.
For a graph $G$, we denote the vertex set of $G$ by $V(G)$ and the
edge set of $G$ by $E(G)$. For any vertex $u \in V(G)$, 
$N_G(u) = \{v \in V(G)~|~(u,v) \in E(G)\}$. We define $d_{G}(u) :=
|N_G(u)|$. The average degree of $G$ is denoted by $d_{av}(G)$. 

Consider a graph $G$ whose vertices are partitioned into two parts, namely $V_A$ and $V_B$. That is, $V(G) = V_A \uplus V_B$. 
% Let $n = |V(G)|$. 
We shall use $S_B(G)$ to denote the graph with $V(S_B(G)) = V(G)$ and 
$E(S_B(G)) = E(G) \setminus \{(u,v)~|~u,v \in V_B\}$. In other words, $S_B(G)$ 
is obtained from $G$ by making $V_B$ a stable set (or an independent set). Let $C_B(G)$ denote the graph with $V(C_B(G)) = V(G)$ and 
$E(C_B(G)) = E(G) \cup \{(u,v)~|~u,v \in V_B\}$. That is, $C_B(G)$ 
is obtained from $G$ by making $V_B$ a clique. Let $G_B$ denote the subgraph 
of $G$ induced on $V_B$. Analogously, we define $S_A(G), C_A(G)$, and $G_A$. 

Since an interval graph is the intersection graph of closed intervals on the
real line, for every interval graph $I_a$, there exists a function
$f_a:V(I_a)\rightarrow \{X\subseteq\mathbb{R}~|~X\mbox{ is a closed interval}\}$,
such that  $\forall u,v\in V(I_a)$, $(u,v)\in E(I_a)\Leftrightarrow f_a(u)\cap
f_a(v)\not=\emptyset$.
The function $f_a$ is called an \emph{interval representation} of the interval
graph $I_a$. Note that the interval representation of an interval graph 
need not be unique. In a similar way, we call a function $f_b$ a \emph{unit interval representation} 
of unit interval graph $I_b$ if $f_b:V(I_b)\rightarrow 
\{X'\subseteq\mathbb{R}~|~X'\mbox{ is a unit length closed interval}\}$,
such that $\forall u,v\in V(I_b)$, $(u,v)\in E(I_b)\Leftrightarrow f_b(u)\cap
f_b(v)\not=\emptyset$. Given a closed interval $X=[y,z]$, we define $l(X):=y$ and $r(X):=z$. We say that
the interval $X$ has \emph{left end-point} $l(X)$ and \emph{right end-point}
$r(X)$.

Given a graph $G$, a \emph{coloring} $\mathcal{C}$ of $V(G)$ using colors $\chi_1, \ldots , \chi_a$ is a map $\mathcal{C}:V(G)\rightarrow \{\chi_1, \ldots , \chi_a\}$. For each $u \in V(G)$, we shall use $\mathcal{C}(u)$ to denote the color of $u$ in $\mathcal{C}$. 
% For each $j \in [a]$, $colorclass[\chi_j,\mathcal{C}] = \{u \in V(G)~|~\mathcal{C}(u) = \chi_j\}$.  
\subsubsection{Definitions, Notations and Assumptions used in 
Sections \ref{CubRepandColoring} and \ref{CubandDegeneracy}:} Recall that the 
degeneracy of a graph is the least number $k$ such that it has a 
vertex enumeration in which each vertex is succeeded by at most $k$ of its neighbors. 
Such an enumeration is called the degeneracy order. 
The graph $G$ that 
we consider in these sections is a $k$-degenerate graph having $V(G) = \{v_1,\ldots  ,v_n\}$, 
$|E(G)| = m$ and $\overline{m}$ ($= {n\choose 2} - m$) denotes the number of non-edges in $G$. 
The enumeration $v_1, \ldots , v_n$ 
is a degeneracy order of $V(G)$ and is denoted by $ \mathcal{D}$. For every $v_i,v_j \in V(G)$, 
we say $ v_i <_{\mathcal{D}} v_j$ if $v_i$ comes before $v_j$ in $\mathcal{D}$ i.e., 
$v_i <_{\mathcal{D}} v_j$ if and only if $i<j$. 
% Therefore, $v_1 <_{\mathcal{D}} v_2 <_{\mathcal{D}} \cdots <_{\mathcal{D}} v_n$. 
Suppose $v_i <_{\mathcal{D}} v_j$. If $(v_i,v_j) \in E(G)$, then we call $v_j$ a \emph{forward neighbor} of $v_i$ and $v_i$ 
is referred to as a \emph{backward neighbor} of $v_j$. Observe that since $G$ is $k$-degenerate, 
a vertex can have at most $k$ forward neighbors. If $(v_i,v_j) \notin E(G)$, 
then $v_j$ a \emph{forward non-neighbor} of $v_i$ and $v_i$ 
is a \emph{backward non-neighbor} of $v_j$.
For any $u \in V(G)$, 
$N_G^f(u) = \{w \in V(G)~|~w \mbox{ is a forward neighbor of } u\}$ and 
$N_G^b(u) = \{w \in V(G)~|~w \mbox{ is a backward neighbor of } u\}$. 
% In other words, for any $i,j \in [n]$, $v_i <_{\mathcal{D}} v_j$ if and only if $i<j$. 
\\
\textbf{Support sets of a non-edge: }For each $(v_x,v_y) \notin E(G)$, where $v_x<_{\mathcal{D}} v_y$, let 
$S_{xy} = \{v_z \in N_G^f(v_x)~|~ v_y <_{\mathcal{D}} v_z\} \cup \{v_y\}$. We call 
$S_{xy}$ the \emph{weak support set} of the non-edge $(v_x,v_y)$. 
Define $T_{xy} = S_{xy} \cup \{v_x\}$. We call 
$T_{xy}$ the \emph{strong support set} of the non-edge $(v_x,v_y)$. 
Let $\mathcal{C}$ be a coloring (need not be proper) of $V(G)$. 
We say $S_{xy}$ is \emph{favorably colored} in $\mathcal{C}$, if $\mathcal{C}(v_y) 
\neq 
\mathcal{C}(v_w)$, $\forall v_w \in S_{xy}\setminus \{v_y\}$. 
% For any $(v_x,v_y) \notin E(G)$, where $v_x<_{\mathcal{D}} v_y$, let 
% $T_{xy} = S_{xy} \cup \{v_x\}$. 
% Let $\mathcal{C}$ be a coloring (need not be proper) of $V(G)$. 
We say $T_{xy}$ is \emph{favorably colored} in $\mathcal{C}$, if $\mathcal{C}(v_y) \neq 
\mathcal{C}(v_w)$, $\forall v_w \in T_{xy}\setminus \{v_y\}$

\section{Cube Representation and Coloring}
\label{CubRepandColoring}
\begin{lemma}
\label{coloringAndCubiLemma}
Let $G$ be a $k$-degenerate graph. 
% For any $(v_x,v_y) \notin E(G)$, where $v_x<_{\mathcal{D}} v_y$, let 
% $T_{xy} = N_f(v_x) \cup \{v_x,v_y\}$. 
Let $\chi = \{\chi_1, \ldots \chi_a\}$ be a set of 
colors and let $\mathbb{C} = \{\mathcal{C}_1, \ldots , \mathcal{C}_b\}$ be a 
family of colorings (need not be proper) of $V(G)$, where each $\mathcal{C}_i$ 
uses colors from the set $\chi$. 
% If for every $T_{xy}$ there exists a coloring $\mathcal{C}_i$, where 
% $i \in [b]$, such that $color(v_y,\mathcal{C}_i) \neq \mathcal{C}_i(v_w)$, 
% $\forall v_w \in T_{xy}\setminus \{v_y\}$, then $\cubi(G) \leq ab$.   
If the strong support set $T_{xy}$ of every non-edge $(v_x,v_y) \notin E(G)$, 
$v_x <_{\mathcal{D}} v_y$, is favorably colored in some 
$\mathcal{C}_i$, where $i \in [b]$, then $\cubi(G) \leq ab$.
\end{lemma}
\begin{proof}
We prove this by constructing $ab$ unit interval graphs $I_{i,j}$ on the vertex set $V(G)$, 
where $i\in [a]$ and 
$j \in [b]$, such that $G = \bigcap_{i=1}^{a} \bigcap_{j=1}^{b} I_{i,j}$. Then the statement will 
follow from Lemma \ref{lemrobertscub}. Let $f_{i,j}$ denote 
a unit interval representation of $I_{i,j}$. 
Let us partition the vertices of $I_{i,j}$ into two parts, namely $A^{ij}$ and $B^{ij}$, 
where $A^{ij} = \{v \in V(G)~|~\mathcal{C}_i(v) = \chi_j\}$ and $B^{ij} = V(G) \setminus A^{ij}$.
For every $i \in [a]$ and $j \in [b]$, a unit interval representation 
$f_{i,j}$ of $I_{i,j}$ is constructed from the coloring $\mathcal{C}_i$ in the following way. 
% \\ Let $colorclass_i[\chi_j] = \{v_{l_1}, v_{l_2}, \ldots , v_{l_r}\}$, 
% where $r \in [n]$ and $v_{l_1} <_{\mathcal{D}} v_{l_2} <_{\mathcal{D}} \cdots <_{\mathcal{D}} v_{l_r}$ i.e., $l_1 < l_2 < \cdots < l_r$. 
For every $v_y \in V(G)$, \\
\begin{eqnarray*}
\mbox{If }v_y \in A^{ij} & \mbox{, then } & \\
% & \mbox{let }v=v_y^{ij},\mbox{ where }y \in [r]\mbox{, and } \\
& & f_{i,j}(v_y) = [y + n, y + 2n] \\
& \mbox{ else } & \\
& & f_{i,j}(v_y) = [g_{max}^{ij}(v_y), g_{max}^{ij}(v_y) + n] \mbox{, where } \\ 
& & g_{max}^{ij}(v_y) = \max(\{g~|~(v_y,v_g) \in E(G), \\
& & v_g \in A^{ij}\} \cup \{0\}).
% $r(v,f_{i,j}) = g_{max}^{ij}(v) + n$, 
% where  $g_{max}^{ij}(v) = \max(\{g \in [r]~|~(v,v_g^{ij} \in E(G)\} \cup \{0\}). 
\end{eqnarray*}  

Since the length of $f_{i,j}(v_y)$ is $n$, for every $v_y \in V(G)$, $I_{i,j}$ is a unit interval graph. 
% From the construction, it is easy to see that $I_{i,j}$ is a unit interval graph 
% because the length of every interval in $f_{i,j}$ is $n$. 
% Let us partition the vertices of 
% $I_{i,j}$ into two parts, namely $A^{ij}$ and $B^{ij}$, where $A^{ij} = colorclass_i[\chi_j]$
% and $B^{ij} = V(I_{i,j}) \setminus A^{ij}$. 
It is easy to see that, $\forall v_x,v_y \in A^{ij}$, 
$2n \in f_{i,j}(v_x) \cap f_{i,j}(v_y)$ and therefore $A^{ij}$ forms a clique in 
$I_{i,j}$. Since $n \in f_{i,j}(v_x) \cap f_{i,j}(v_y)$, $\forall v_x,v_y \in B^{i,j}$, 
$B^{i,j}$ too forms a clique in $I_{i,j}$. For every $(v_x,v_y) \in E(G)$, with 
$v_x \in A^{ij}$ and $v_y \in B^{ij}$, we have 
$l(f_{i,j}(v_y)) = g_{max}^{ij}(v_y) \leq n \leq l(f_{i,j}(v_x)) = n+x \leq n + g_{max}^{ij}(v_y)$, where 
the last inequality is inferred from the fact that $(v_x,v_y) \in E(G)$ and $v_x \in 
A^{ij}$. But 
$n + g_{max}^{ij}(v_y) = r(f_{i,j}(v_y))$. Therefore, we get $l(f_{i,j}(v_y)) \leq l(f_{i,j}(v_x)) 
\leq r(f_{i,j}(v_y))$ and hence $(v_x,v_y) \in E(I_{i,j})$. Hence $I_{i,j}$ is a supergraph 
of $G$. 
% $(x+n) \in f_{i,j}(v_x) \cap f_{i,j}(v_y)$ 
% because $l(f_{i,j}(v_y)) = g_{max}^{ij}(v_y) \leq n \leq n+x = l(f_{i,j}(v_x)) \leq n + g_{max}^{ij}(v_y) = 
% r(f_{i,j}(v_y))$.  
% % \leq n + n \leq 2n + x = r(f_{i,j}(v_y))$. 
% Hence $I_{i,j}$ is a supergraph of $G$. 

Let $v_x<_{\mathcal{D}} v_y$ and $(v_x,v_y) \notin E(G)$. 
We now have to show that there exists some unit interval graph $I_{i,j}$ such 
that $(v_x,v_y) \notin E(I_{i,j})$. We know that, by assumption, there exists a 
coloring, say $\mathcal{C}_i$ (where $i \in [a]$), such that the strong 
support set $T_{xy}$ is 
favorably colored in $\mathcal{C}_i$. 
Let $\chi_j = \mathcal{C}_i(v_y)$. 
% Then $x \notin colorclass_i[\chi_j], 
% N_f(x) \cap colorclass_i[\chi_j] = \emptyset$ and $y \in colorclass_i[\chi_j]$. 
Let $g = g_{max}^{ij}(v_x)$. 
We claim that $g<y$. Assume, for contradiction, that $g>y$. 
Then $g \neq 0$ and $v_g \in A^{ij}$. Since $y>x$, we get 
$g>x$. Therefore, $v_g \in N_G^f(v_x)$ and $g>y$. This implies that 
$v_g \in T_{xy}$. Since $T_{xy}$ is favorably colored in 
$\mathcal{C}_i$, $\mathcal{C}_i(v_g) \neq \chi_j$. This contradicts 
the fact that $v_g \in A^{ij}$. Thus we prove the claim. 
% Since $N_G^f(v_x) \cap A^{ij} = \emptyset$, 
% $v_g <_{\mathcal{D}} v_x$ i.e., $g < x$. As $x<y$, we have $g<y$. 
Therefore, $r(f_{i,j}(v_x)) = n+g < n+ y = 
l(f_{i,j}(v_y))$ and hence $(v_x,v_y) \notin E(I_{i,j})$. We infer that  
$G = \bigcap_{i=1}^{a} \bigcap_{j=1}^{b} I_{i,j}$.
% 
% Let $colorclass_i[\chi_j] = \{v_1^{ij}, v_2^{ij}, \ldots , v_r^{ij}\}$, where 
% $r<n$ and $v_1^{ij}<_{\mathcal{D}} v_2^{ij}<_{\mathcal{D}} \cdots <_{\mathcal{D}} v_r^{ij}$. Let $y = v_h^{ij}$, where $h \in [r]$, and let $g = g_{max}^{ij}(x)$. 
% If $g=0$, then $g<h$. Suppose $g \neq 0$. 
% Then since $N_f(x) \cap colorclass_i[\chi_j] = \emptyset$, $v_g^{ij} <_{\mathcal{D}} x$. 
% Since $x <_{\mathcal{D}} y = v_h^{ij}$, we get $g<h$. Therefore, $r(f_{i,j}(x)) = n + g_{max}^{ij}(x) = n+g < n+ h = 
% l(f_{i,j}(v_h^{ij})) = l(f_{i,j}(y))$. Therefore, $(x,y) \notin E(I_{i,j})$. Hence we prove that 
% $G = \bigcap_{i=1}^{a} \bigcap_{j=1}^{b} I_{i,j}$. 
\begin{remark}
 \label{coloringAndCubiRem}
Note that $\forall v\in V(G), i \in [a], j \in [b]$, either $f_{i,j}(v) \cap [n,n] \neq \emptyset$ or $f_{i,j}(v) \cap [2n,2n] \neq \emptyset$ or both. 

\end{remark}

\hfill \qed
\end{proof}

\section{Cubicity and Degeneracy}
\label{CubandDegeneracy}

\subsection{An Upper Bound - Probabilistic Approach}
% Here we outline a probabilistic proof for cubicity of $k$-degenerate graphs 
% that yields a bound with a slightly better constant. 
\begin{theorem}
\label{probabilisticTheorem}
For every $k$-degenerate graph $G$, $\cubi(G) \leq (k+2) \cdot \lceil 2e \log n \rceil$ 
\end{theorem}
\begin{proof}
Let $\chi = \{\chi_1, \ldots , \chi_{k+2}\}$ be a set of $k+2$ colors. Generate 
a random coloring $\mathcal{C}_1$ (need not be a proper coloring) of vertices of $G$ in the 
following way: For each vertex $v_x \in V(G)$, pick a color $\chi_j$, where $j \in [k+2]$, 
uniformly at random from $\chi$ and set $\mathcal{C}_1(v_x) = \chi_j$. In a 
similar way, independently generate random colorings $\mathcal{C}_2, 
\ldots , \mathcal{C}_b$, where $b= \lceil 2e \log n \rceil$. 

For every $(v_x,v_y) \notin E(G)$ and $v_x<_{\mathcal{D}} v_y$, since $G$ is $k$-degenerate 
% let $T_{xy}$ denote the set defined in 
% Lemma \ref{coloringAndCubiLemma}. 
we have $|T_{xy}| = t \leq k+2$. 
% \begin{eqnarray*}
% Pr[color(v_y,\mathcal{C}_1) \neq color(v_w,\mathcal{C}_1,), \forall v_w \in T_{xy} 
% \setminus \{v_y\} ] & \geq & \frac{(k+2) (k+1)^{(k+1)}}{(k+2)^{(k+2)}} \\
% & = & \left(\frac{k+1}{k+2}\right)^{k+1}.
$Pr[T_{xy} \mbox{ is favorably colored in } \mathcal{C}_i ] 
 =  \frac{(k+2)(k+1)^{t-1}}{(k+2)^{t-1}} 
 =   \left(\frac{k+1}{k+2}\right)^{t-1}$ 
 $\geq  \left(\frac{k+1}{k+2}\right)^{k+1}.  
$
% \end{eqnarray*}
% \begin{eqnarray*}
$\mbox{Therefore, }Pr[T_{xy} \mbox{ is not favorably colored in } 
\mathcal{C}_i ] \leq 1 - \left(\frac{k+1}{k+2}\right)^{k+1}$ 
$\leq e^{-\left(\frac{k+1}{k+2}\right)^{k+1}}.
$
% \[Pr[ \bigcap_{i=1}^{p} (\exists v_{w_i} \in T_{xy} \setminus \{v_y\} \mbox{ such that } color(v_y, \mathcal{C}_i) 
% = color(v_{w_i}, \mathcal{C}_i))] \leq e^{-p\left(\frac{k+1}{k+2}\right)^{k+1}}. 
% \]
% Pr[\bigcap_{i=1}^{b} (T_{xy} \mbox{ is not favorably colored in } \mathcal{C}_i) ] & \leq & e^{-b\left(\frac{k+1}{k+2}\right)^{k+1}}. 
% \end{eqnarray*}
Now taking $b=\lceil 2e \log n \rceil$, 
\begin{eqnarray*}
Pr [ \bigcup_{x,y: (v_x<_{\mathcal{D}} v_y), ((v_x,v_y) \notin E(G))} \bigcap_{i=1}^{b} (T_{xy} 
\mbox{ is not favorably colored in } \mathcal{C}_i)] \\ 
\leq n^2 e^{-b\left(\frac{k+1}{k+2}\right)^{k+1}} < 1. 
\end{eqnarray*}

Hence, 
% \begin{eqnarray*}
$Pr[\mathcal{C}_1, \ldots, \mathcal{C}_b \mbox{ satisfy 
the condition of Lemma \ref{coloringAndCubiLemma}] } >0.$
% \end{eqnarray*} 
% Putting RHS of Inequality (\ref{ineq1}) to be greater that $0$ 
% % we get $1 - n^2 e^{-p\left(\frac{k+1}{k+2}\right)^{k+1}} > 0$ \\ $\implies  
% and solving for $b$, we get 
% $b > (\frac{k+2}{k+1})^{k+1} 2 \log n$. That is, when $b  = 
% 2e \log n > (1 + \frac{1}{k+1})^{k+1} 2 \log n = (\frac{k+2}{k+1})^{k+1} 2 \log n$ 
% there is a positive probability for colorings $\mathcal{C}_1, \mathcal{C}_2, \ldots \mathcal{C}_b$ to 
% satisfy the condition of Lemma \ref{coloringAndCubiLemma}. 
Therefore, there exists a 
coloring $ \mathcal{C}_1, \ldots \mathcal{C}_b$, with $b= \lceil 2e \log n \rceil$,  
of $V(G)$ using colors from the set $\{\chi_1, 
\ldots , \chi_{k+2}\}$ such that the condition of Lemma \ref{coloringAndCubiLemma} is satisfied. 
Hence by Lemma \ref{coloringAndCubiLemma},  $\cubi(G) \leq  (k+2) \cdot \lceil 2e \log n \rceil$. 
\hfill \bbox
\end{proof}
\begin{corollary}
\label{ProbabilisticThmCor}
Let $G$ be a $k$-degenerate graph with $n$ vertices and $G'$ a graph constructed from $G$ with $V(G') = V(G) \cup V'$, $E(G') = E(G) \cup \{(u,v)~|~u \in V', v \in V(G') \}$, and $V(G) \cap V' = \emptyset$. Then, $\cubi(G') \leq (k+2)\cdot \lceil 2e \log n \rceil$.  
\end{corollary}
\begin{proof}
From Theorem \ref{probabilisticTheorem}, we know that there exist $(k+2)\cdot \lceil 2e \log n \rceil$ unit interval graphs $I_{i,j}$, where $i \in [k+2]$, $j \in [\lceil 2e \log n \rceil]$, such that $G = \bigcap_i\bigcap_jI_{i,j}$. Let $f_{i,j}$ be the unit interval representation of each $I_{i,j}$ as per the construction in Lemma \ref{coloringAndCubiLemma}. We now construct $(k+2)\cdot \lceil 2e \log n \rceil$ unit interval graphs $I'_{i,j}$, where $i \in [k+2]$, $j \in [\lceil 2e \log n \rceil]$, such that $G' = \bigcap_i\bigcap_jI'_{i,j}$. Let $f'_{i,j}$ be a unit interval representation of $I'_{i,j}$. Then for each $v \in V(G')$, 
\begin{eqnarray*}
f'_{i,j}(v) & = & f_{i,j}(v) \mbox{, if }v \notin V' \\
f'_{i,j}(v) & = & [n,2n] \mbox{, if }v \in V' 
\end{eqnarray*}
From Remark \ref{coloringAndCubiRem} in Lemma \ref{coloringAndCubiLemma}, every $v \in V'$ is adjacent with every other vertex in each $I'_{i,j}$ since $f'_{i,j}(v) = [n,2n],~\forall v \in V'$. 
\hfill \bbox
\end{proof}
\subsubsection{Tightness of Theorem \ref{probabilisticTheorem}}
Recall that,  a realizer of a partially ordered set (poset) $\mathcal{P}$ is a set of linear extensions of $\mathcal{P}$, say $\mathcal{R}$, which satisfy the following condition: for any two distinct elements $x$ and $y$, $x\leq y$ in $\mathcal{P}$ if and only if $x \leq y$ in $L$, $\forall L \in \mathcal{R}$.  The poset dimension of $\mathcal{P}$, denoted by $dim(\mathcal{P})$, is the minimum integer $k$ such that there exists a realizer of $\mathcal{P}$ of cardinality $k$. Let $G_{\mathcal{P}}$ denote the underlying comparability graph of $\mathcal{P}$. Then by Theorem $1$ in \cite{DiptAdiga}, $\boxi(G_{\mathcal{P}}) \geq \frac{dim(\mathcal{P})}{2}$.  

Let $\mathbb{P}(n,p)$ be the probability space of height-2 posets with $n$ minimal elements forming set $A$ and $n$ maximal elements forming set $B$, where for any $a \in A$ and $b \in B$, $Pr[a<b] = p$. 
% The following theorem is due to Erd\H{o}s, Kierstead, and Trotter \cite{erdosKierstead}. 
% \begin{theorem}
%  For every $\epsilon > 0$, there exists $\delta > 0$ so that if $\frac{\log^{1+\epsilon}n }{n} < p < 1 - n^{\epsilon - 1}$, then, $dim(\mathcal{P}) > \frac{\delta pn \log(pn)}{1+\delta p \log(pn)}$ for almost all $\mathcal{P} \in \mathbb{P}(n,p)$. 
% \end{theorem}
Erd\H{o}s, Kierstead, and Trotter in \cite{erdosKierstead} proved that 
when $p=\frac{1}{\log n}$, for almost all posets $\mathcal{P} \in \mathbb{P}(n,p)$, $\Delta(G_{\mathcal{P}}) < \frac{\delta_1 n}{\log n}$ and $dim(\mathcal{P}) > \delta_2 n$, where $\delta_1$ and $\delta_2$ are some positive constants. Then by Theorem $1$ in \cite{DiptAdiga}, $\cubi(G_{\mathcal{P}}) \geq \boxi(G_{\mathcal{P}}) \geq  \frac{dim(\mathcal{P})}{2} \geq \frac{\delta_2 n}{2}$. 

We know that $G_{\mathcal{P}}$ is $\Delta(G_{\mathcal{P}})$-degenerate. By Theorem \ref{probabilisticTheorem},  $\cubi(G_{\mathcal{P}}) \leq (\Delta(G_{\mathcal{P}}) + 2)\cdot \lceil 2e \log n \rceil  \leq (\frac{\delta_1 n}{\log n} + 2)\cdot \lceil 2e \log n \rceil \leq cn$, where $c$ is some constant. Hence the upper bound for cubicity given in Theorem \ref{probabilisticTheorem} is  tight.

\subsection{Deterministic Algorithm}
CONSTRUCT\_CUB\_REP($G$) is a deterministic algorithm which takes a simple, finite $k$-degenerate 
graph $G$ as input and outputs a cube representation in $8k\alpha$ dimensional 
space i.e., $8k\alpha$ unit interval graphs $I_{1,1}, \ldots , I_{1,8k}, 
\ldots , I_{\alpha, 1}, \ldots , I_{\alpha, 8k}$ such 
that $G = \bigcap_{i=1}^{\alpha} \bigcap_{j=1}^{8k} I_{i,j}$. In order 
to achieve this, CONSTRUCT\_CUB\_REP\\($G$) invokes the procedure CONSTRUCT\_COLORING i.e., Algorithm \ref{construct_coloring_abridged} 
(for a detailed version of this procedure , see Algorithm 4.4), 
$\alpha$ 
times and thereby generates $\alpha$ colorings 
$\mathcal{C}_1, \ldots , \mathcal{C}_{\alpha}$, where each coloring 
uses colors from the set $\{\chi_1, \ldots , \chi_{8k} \}$. 
Then from each coloring $\mathcal{C}_i$, it constructs $8k$ 
unit interval graphs $I_{i,1}, \ldots , I_{i,8k}$ using the construction 
described in Lemma \ref{coloringAndCubiLemma}, which is implemented in procedure 
CONSTRUCT\_UNIT\_INTERVAL\_GRAPHS.

\newcounter{line}
\newcommand{\MYSTATE}[1]{\STATE{\refstepcounter{line}\theline. #1}}

\begin{algorithm}[h]
\label{det_algo}
\begin{algorithmic}
\caption{CONSTRUCT\_CUB\_REP(G)}
\setcounter{line}{0}
\FOR{$y =n$ to $1$}
\MYSTATE{Initialize $BNN_1[v_y] \leftarrow \{v_x \in V(G)~|~v_x <_{\mathcal{D}} v_y, 
(v_x,v_y) \notin E(G)\}$.}
\MYSTATE{Initialize $FNN_1[v_y] \leftarrow \{v_z \in V(G)~|~v_y <_{\mathcal{D}} v_z, 
(v_y,v_z) \notin E(G)\}$.}
\ENDFOR
\MYSTATE{SET FLAG $\leftarrow$ TRUE.}
\MYSTATE{SET i $\leftarrow$ 0.}
\WHILE{FLAG = TRUE}
\MYSTATE{i++.}
\MYSTATE{$\mathcal{C}_i = $ CONSTRUCT\_COLORING($i$)}.
\FOR{$y=1$ to $n$}
\MYSTATE{SET $BNN_{i+1}[v_y] \leftarrow BNN_i[v_y] \setminus W(v_y,\mathcal{C}_i)$}
\MYSTATE{SET $FNN_{i+1}[v_y] \leftarrow FNN_i[v_y] \setminus Y(v_y,\mathcal{C}_i)$}
\ENDFOR
\MYSTATE{If $FNN_{i+1}[v_y] = \emptyset$, $\forall v_y \in V(G)$, then FLAG = FALSE.}
\ENDWHILE
\MYSTATE{SET $\alpha \leftarrow i$}
\MYSTATE{CONSTRUCT\_UNIT\_INTERVAL\_GRAPHS()}
% \MYSTATE{Return $f_{i,j}(v_y)$, for all $i\in [\alpha], j \in [8k], y \in [n]$ }
\end{algorithmic}
\end{algorithm}

\begin{algorithm}[h]
\label{construct_coloring_abridged}
/*For a detailed version of this procedure, see Algorithm 4.4. \\
All data structures are assumed to be global. \\
\textbf{Notational Note:} \\
Let $\mathcal{C}_i^{v_z}$ denote the partial coloring at the stage 
when we have colored the vertices $v_n$ to $v_z$. 
Let $\mathcal{C}_i^{v_z = \chi_c}$ denote the partial coloring that 
results if we extend $\mathcal{C}_i^{v_{z+1}}$ by assigning color $\chi_c$ to 
$v_z$.*/
\begin{algorithmic}
\caption{CONSTRUCT\_COLORING($i$) /* abridged */}
\setcounter{line}{0}
\FOR{$y=n$ to $1$}
\FOR{each $\chi_c \in \{\chi_1, \ldots ,\chi_{8k}$ \}}
% \MYSTATE{Let $\mathcal{C}_i^{v_y = \chi_c}$ denote the partial coloring that 
% results if we extend $\mathcal{C}_i^{v_{y+1}}$ by assigning color $\chi_c$ to 
% $v_y$.}
\MYSTATE{Compute $|X(v_y,\mathcal{C}_i^{v_y = \chi_c})|$, 
$|Y(v_y,\mathcal{C}_i^{v_y = \chi_c})|$, and $|Z(v_y,\mathcal{C}_i^{v_y = \chi_c})|$ 
as per equations (\ref{Xeqn}),(\ref{Yeqn}), and (\ref{Zeqn}) respectively.}
\IF{$|X(v_y,\mathcal{C}_i^{v_y = \chi_c})| \geq \frac{3}{4}|BNN_i[v_y]|$ and 
$|Y(v_y,\mathcal{C}_i^{v_y = \chi_c})| \geq \frac{3}{4}|Z(v_y,\mathcal{C}_i^{v_y = \chi_c})|$} 
% \COMMENT{Such a $c$ always exists. See Lemma \ref{ratiolemma}}
\MYSTATE{SET $\mathcal{C}_i^{v_y} \leftarrow \mathcal{C}_i^{v_y= \chi_c}$ (i.e. 
SET $\mathcal{C}_i(v_y) \leftarrow \chi_c$).} 
\MYSTATE{SET $Y(v_y,\mathcal{C}_i^{v_y}) \leftarrow  Y(v_y,\mathcal{C}_i^{v_y= \chi_c})$}
\MYSTATE{BREAK.}
\ENDIF
\ENDFOR
% \MYSTATE{$X(v_y)$}
\ENDFOR
\FOR{$y=1$ to $n$}
\MYSTATE{Compute $W(v_y,\mathcal{C}_i)$ as per equation (\ref{Weqn})}
\MYSTATE{SET $Y(v_y,\mathcal{C}_i) \leftarrow Y(v_y,\mathcal{C}_i^{v_1})$}
\ENDFOR
\MYSTATE{Return $\mathcal{C}_i$. }
\end{algorithmic}
\end{algorithm}

\begin{algorithm}
\label{construct_Unit_Interval_Graphs}
/*All data structures are assumed to be global. */ 
\begin{algorithmic}
\caption{CONSTRUCT\_UNIT\_INTERVAL\_GRAPHS()}
\setcounter{line}{0}
\MYSTATE{INITIALIZE $l(f_{i,j}(v_y)) \leftarrow 0, r(f_{i,j}(v_y)) \leftarrow n, 
\forall y \in [n], i \in \alpha, j \in [8k]$}
\FOR{ $i=1$ to  $\alpha$}
\FOR{ $y=n$ to  $1$}
\MYSTATE{SET $j \leftarrow c$, such that $\mathcal{C}_i(v_y) = \chi_c$}
\MYSTATE{SET $l(f_{i,j}(v_y)) \leftarrow y + n$}
\MYSTATE{SET $r(f_{i,j}(v_y)) \leftarrow y + 2n$}
\FOR{each $v \in N_G^b(v_y)$}
\IF{$(\mathcal{C}_i(v) \neq j) \cap (l(f_{i,j}(v)) = 0)$}
\MYSTATE{SET $l(f_{i,j}(v)) \leftarrow y$}
\MYSTATE{SET $r(f_{i,j}(v)) \leftarrow y + n$}
\ENDIF
\ENDFOR
\ENDFOR
\ENDFOR
\MYSTATE{Output $f_{i,j}(v_y), \forall y \in [n], i \in \alpha, j \in [8k]$}
\end{algorithmic}
\end{algorithm}

Note that in order for $G$ to be equal to $\bigcap_{i=1}^{\alpha} \bigcap_{j=1}^{8k} I_{i,j}$, 
Lemma \ref{coloringAndCubiLemma} requires that the colorings $\mathcal{C}_1, \ldots , 
\mathcal{C}_{\alpha}$ satisfy the following property: 
for every $(v_x,v_y) \notin E(G)$, where $v_x <_{\mathcal{D}} v_y$, there exists an $i \in [\alpha]$ 
such that the strong support set $T_{xy}$ of this non-edge is favorably colored in $\mathcal{C}_i$. 
The colorings $\mathcal{C}_1, \ldots ,\mathcal{C}_{\alpha}$ are generated one by one keeping 
this objective in mind. At the stage when we have just generated the $(i-1)$-th coloring 
$\mathcal{C}_{i-1}$, if a non-edge $(v_x,v_y)$ is such that its strong support set $T_{xy}$ 
is already favorably colored in some $\mathcal{C}_j$, where $j < i$, then we say that the 
non-edge $(v_x,v_y)$ is already DONE. Naturally at each stage we have to keep track of the 
non-edges that are not yet DONE. In order to do this, we introduce 
two data structures $BNN_i$ and $FNN_i$, for all $i \in [\alpha]$ \footnote{$BNN$ - 
Backward Non-Neighbor, $FNN$ - Forward Non-Neighbor}. For 
each $v_y \in V(G)$, 
\begin{eqnarray*}
BNN_i[v_y] & = & \{ v_x \in V(G)~|~ v_x \mbox{ is a backward non-neighbor of } 
v_y \mbox{, and } (v_x,v_y) \\
& &  \mbox{ is not yet DONE with respect to } \mathcal{C}_1, \ldots , \mathcal{C}_{i-1}.\} \\
FNN_i[v_y] & = & \{ v_z \in V(G)~|~ v_z \mbox{ is a forward non-neighbor of } 
v_y \mbox{, and }  (v_y,v_z) \\
& &  \mbox{ is not yet DONE with respect to } \mathcal{C}_1, \ldots , \mathcal{C}_{i-1}.\} \\
\end{eqnarray*}
It is easy to see that, $\bigcup_{v_y \in V(G)} BNN_i[v_y] = 
\bigcup_{v_y \in V(G)} FNN_i[v_y]$ and therefore, \\
$\left(\bigcup_{v_y \in V(G)}BNN_i[v_y] = \emptyset \right)$  $\iff$ 
$\left(\bigcup_{v_y \in V(G)} FNN_i[v_y] = \emptyset \right)$.
In Theorem \ref{AlgoTheorem}, we show that if we select $\alpha$ to be at least 
$(\lceil 2.42 \log n\rceil + 1)$, 
then $FNN_{\alpha+1}[v_y] = \emptyset$, $\forall v_y \in V(G)$. 
This clearly would mean that all non-edges are DONE with respect to 
$\mathcal{C}_1, \ldots , \mathcal{C}_{\alpha}$. In other words, the condition of Lemma 
\ref{coloringAndCubiLemma} 
will be satisfied for $\mathcal{C}_1, \ldots , \mathcal{C}_{\alpha}$. 

The only thing that remains to be discussed now is how our coloring strategy (i.e. the procedure 
CONSTRUCT\_COLORING) achieves the above objective, namely \\
$BNN_{\alpha+1}[v_y]= \emptyset$ and $FNN_{\alpha+1}[v_y] = \emptyset$, $\forall v_y \in V(G)$, if 
 $\alpha \geq (\lceil 2.42 \log n\rceil + 1)$. To start with $BNN_1[v_y]$ (respectively $FNN_1[v_y]$) contains 
all the backward (respectively forward) non-neighbors of $v_y$. The procedure 
CONSTRUCT\_COLORING\\($i$) generates the $i$-th coloring $\mathcal{C}_i$ as 
follows. It colors vertices in the reverse degeneracy order starting from vertex 
$v_n$. The partial coloring at the stage when we have colored the vertices 
$v_n$ to $v_z$ is denoted by $\mathcal{C}_i^{v_z}$. Note that $\mathcal{C}_i^{v_1} = \mathcal{C}_i$. 
Consider the stage at which the algorithm has already colored the vertices from $v_n$ up to $v_{y+1}$ 
and is about to color $v_y$. That is, we have the partial coloring $\mathcal{C}_i^{v_{y+1}}$ and 
are about to extend it to the partial coloring $\mathcal{C}_i^{v_y}$ by assigning 
one of the $8k$ possible colors to vertex $v_y$. 
% Clearly, there are $8k$ possible colors for $v_y$. 
% If we decide to color $v_y$ with color $\chi_c$, let us call the resultant coloring $\mathcal{C}_i^{v_y = \chi_c}$. 
Let $\mathcal{C}_i^{v_y = \chi_c}$ denote the partial coloring that results if we extend 
$\mathcal{C}_i^{v_{y+1}}$ by assigning color $\chi_c$ to $v_y$. 
The coloring $\mathcal{C}_i$ and the partial colorings $\mathcal{C}_i^{v_z}$, $\forall v_z \in V(G)$ 
and $\mathcal{C}_i^{v_z= \chi_c}$, $\forall v_z \in V(G), \chi_c \in \{\chi_1, \ldots , \chi_{8k}\}$, 
will be generically called \emph{the colorings associated with the $i$-th stage} (
i.e. the $i$-th invocation of CONSTRUCT\_COLORING).

With respect to colorings $\mathcal{C}_1, \ldots , \mathcal{C}_{i-1}$ and some coloring 
$\mathcal{C}_i'$ associated with the $i$-th stage, we define the following sets: 
\begin{eqnarray}
\label{Weqn}
W(v_w,\mathcal{C}_i') & = & \{v_x \in BNN_i[v_w]~|~\mbox{the strong support set } T_{xw} 
\mbox{ of non-edge } \\   
& & (v_x,v_w) \mbox{ is favorably colored in }  \mathcal{C}_i'\} \nonumber \\
\label{Xeqn}
X(v_w,\mathcal{C}_i') & = & \{v_x \in BNN_i[v_w]~|~\mbox{the weak support set } S_{xw} 
\mbox{ of non-edge } \\ 
& & (v_x,v_w) \mbox{ is favorably colored in }  \mathcal{C}_i'\} \nonumber \\
\label{Yeqn}
Y(v_w,\mathcal{C}_i') & = & \{v_z \in FNN_i[v_w]~|~\mbox{the strong support set } T_{wz} 
\mbox{ of non-edge } \\ 
& & (v_w,v_z) \mbox{ is favorably colored in }  \mathcal{C}_i'\} \nonumber \\
\label{Zeqn}
Z(v_w,\mathcal{C}_i') & = & \{v_z \in FNN_i[v_w]~|~\mbox{the weak support set } S_{wz} 
\mbox{ of non-edge } \\ 
& & (v_w,v_z) \mbox{ is favorably colored in }  \mathcal{C}_i'\}  \nonumber
\end{eqnarray}

Naturally, we want to give a color $\chi_c$ to $v_y$ such that 
a large number of (not yet DONE) non-edges incident on $v_y$ get DONE. 
With respect to the colorings $\mathcal{C}_1, \ldots , \mathcal{C}_{i-1}$ 
and the partial coloring $\mathcal{C}_i^{v_y = \chi_c}$, we define the 
status of a non-edge incident on $v_y$ as follows: A non-edge 
$(v_y,v_z) \in FNN_i[v_y]$ is DONE\footnote{Recall that we had defined earlier 
that a non-edge $(v_x,v_y)$ is DONE with respect to a list of colorings 
$\mathcal{C}_1, \ldots , \mathcal{C}_{i-1}$ if $T_{xy}$ was favorably 
colored in some $C_j$, where $j<i$. Here we extend this notion, 
by allowing the partial coloring $\mathcal{C}_i^{v_y = \chi_c}$ also 
in the list.} 
if $T_{yz}$ is favorably colored in 
$\mathcal{C}_i^{v_y = \chi_c}$ and is NOT-DONE if 
$T_{yz}$ is not favorably colored in $\mathcal{C}_i^{v_y = \chi_c}$. 
A non-edge 
$(v_x,v_y) \in BNN_i[v_y]$ is HOPELESS\footnote{A HOPELESS non-edge 
$(v_x,v_y)$ will not be DONE with respect to $\mathcal{C}_1, \ldots , 
\mathcal{C}_{i}$ if we set $\mathcal{C}_i(v_y) = \chi_c$, irrespective of the 
color given to $v_{y-1}, \ldots , v_1$.} if $S_{xy}$ (which happens to be a proper 
subset of $T_{xy}$) is not favorably colored in $\mathcal{C}_i^{v_y = \chi_c}$ 
and is HOPEFUL if $S_{xy}$ is favorably colored in $\mathcal{C}_i^{v_y = \chi_c}$. 
So when we decide a color for $v_y$, our intention is to make a large fraction of 
the HOPEFUL non-edges of $FNN_i[v_y]$ (i.e. the set $Z(v_y, \mathcal{C}_i^{v_y = \chi_c})$), 
DONE and to make a large fraction of 
$BNN_i[v_y]$, HOPEFUL. More formally, 
we want the algorithm to assign a color $\chi_c$ to $v_y$ such that the following 
two conditions are satisfied.\\
(i) $|X(v_y,\mathcal{C}_i^{v_y = \chi_c})| \geq \frac{3}{4}|BNN_i[v_y]|$, and \\
(ii)$|Y(v_y,\mathcal{C}_i^{v_y = \chi_c})| \geq \frac{3}{4}|Z(v_y,\mathcal{C}_i^{v_y = \chi_c})|$. \\
The obvious question then is whether such a color $\chi_c$ 
always exists, for each $v_y \in V(G)$. Lemma \ref{ratiolemma} 
answers this question in the affirmative. It follows that, 
the number of non-edges that are not yet DONE with respect to 
colorings $\mathcal{C}_1, \ldots \mathcal{C}_{i}$ is at most a constant 
fraction of the number of non-edges that were not DONE with respect to 
colorings $\mathcal{C}_1, \ldots \mathcal{C}_{i-1}$. This is 
formally proved in Lemma \ref{mbarlemma}. That 
$BNN_{\alpha+1}[v_y] = \emptyset$ and $FNN_{\alpha+1}[v_y] = \emptyset$, 
$\forall v_y \in V(G)$, is a consequence of this and is formally proved 
in Theorem \ref{AlgoTheorem}.

\begin{lemma}
\label{ratiolemma}
For every $ i \in [\alpha], v_y \in V(G)$, 
(i) $|X(v_y,\mathcal{C}_i)| \geq \frac{3}{4}|BNN_i[v_y]|$, and \\
(ii)$|Y(v_y,\mathcal{C}_i)| \geq \frac{3}{4}|Z(v_y,\mathcal{C}_i)|$.
\end{lemma}
\begin{proof}
The statement of the lemma is obvious if the BREAK statement in 
Step 4 of CONSTRUCT\_COLORING($i$) (abridged version) is executed, 
for every $ i \in [\alpha]$ and $ v_y \in V(G)$. In order to 
prove that the BREAK statement will be executed, it is sufficient to show that 
there exists a color $\chi_c \in \{\chi_1, \ldots , \chi_{8k}\}$ such that 
$|X(v_y,\mathcal{C}_i^{v_y = \chi_c})| \geq \frac{3}{4}|BNN_i[v_y]|$ and 
$|Y(v_y,\mathcal{C}_i^{v_y = \chi_c})| \geq \frac{3}{4}|Z(v_y,\mathcal{C}_i^{v_y = \chi_c})|$.  
Since the vertices in $Z(v_y,\mathcal{C}_i^{v_y = \chi_c})$ or 
$Z(v_y,\mathcal{C}_i)$ do not depend on the colors given to 
$v_1, \ldots v_y$, we have $Z(v_y,\mathcal{C}_i^{v_y = \chi_c}) = 
Z(v_y,\mathcal{C}_i)$ . Hence, 
$Z(v_y,\mathcal{C}_i^{v_y = \chi_c})$ and $Z(v_y,\mathcal{C}_i)$ 
can be used interchangeably. 
% That is, for every $v_z \in FNN_i[v_y]$, $S_{yz}$ is favorably colored in $\mathcal{C}_i^{v_y = \chi_c}$ 
% if and only if $S_{yz}$ is favorably colored in $\mathcal{C}_i$.  
% This is obvious when $v_y = v_n$ as $Z(v_n,\mathcal{C}_i^{v_n = \chi_c}) = Z(v_n,\mathcal{C}_i) = \emptyset$. 
% 
% $v_z \in $
% it 
% This is obvious when $v_y = v_n$ as $Z(v_n,\mathcal{C}_i^{v_n = \chi_c}) = Z(v_n,\mathcal{C}_i) = \emptyset$. 
% Suppose $v_y \neq v_n$. 
% Then, observe that every element of the weak support set $S_{yz}$ of non-edge $(v_y,v_z)$ is already 
% colored in $\mathcal{C}_i^{v_{y+1}}$ and hence $\mathcal{C}_i$.
% Let $A = BNN_i[v_y] \times Z(v_y,\mathcal{C}_i)$. Let 
% $<v_x,v_z>$ be an element of $A$. Pick a color $\chi_c$ uniformly 
% at random from the set $\{\chi_1, \ldots, \chi_{8k}\}$. We say 
% $\chi_c$ is \emph{good} to $<v_x,v_z>$, if $v_x \in X(v_y,\mathcal{C}_i^{v_y = \chi_c})$ and 
% $v_z \in Y(v_y,\mathcal{C}_i^{v_y = \chi_c})$.  
% $Pr[v_x \notin X(v_y,\mathcal{C}_i^{v_y = \chi_c})] = 
% Pr[S_{xy} \mbox{ is not favorably colored in } \mathcal{C}_i^{v_y = \chi_c} ] \leq 
% \frac{1}{8}$. The last step is due to the fact that there are at most $k+1$ vertices  
% in $S_{xy}$ and for $v_y$ to get a color which coincides with the color of another 
% vertex in $S_{xy}$, it has at most $k$ choices from the set $\{\chi_1, \ldots, \chi_{8k}\}$. 
% $Pr[v_z \notin Y(v_y,\mathcal{C}_i^{v_y = \chi_c})] = 
% Pr[T_{yz} \mbox{ is not favorably colored in } \mathcal{C}_i^{v_y = \chi_c} ] \leq $

Let $A = BNN_i[v_y] \times Z(v_y,\mathcal{C}_i)$. Let 
$<v_x,v_z>$ be an element of $A$. We say a color $\chi_c$ is 
\emph{good} for $<v_x,v_z>$, if $v_x \in X(v_y,\mathcal{C}_i^{v_y = \chi_c})$ and 
$v_z \in Y(v_y,\mathcal{C}_i^{v_y = \chi_c})$. In other words,  
$\chi_c$ is \emph{good} for $<v_x,v_z>$, if both $S_{xy}$ and $T_{yz}$ are 
favorably colored in $\mathcal{C}_i^{v_y = \chi_c}$. 
% In other words, 
% a color $\chi_c$ is \emph{good} to $<v_x,v_z>$, if $\mathcal{C}_i(v_y) = \chi_c$ implies 
% $v_x \in X(v_y,\mathcal{C}_i)$ and $v_y \in Y(v_y,\mathcal{C}_i)$. 
$S_{xy}$ is favorably colored in $\mathcal{C}_i^{v_y = \chi_c}$, 
if $\chi_c \notin P$, where $P = \{\mathcal{C}_i^{v_y = \chi_c}(v_w)~|~v_w \in N_G^f(v_x), v_y <_{\mathcal{D}} v_w\}$. 
Since $|N_G^f(v_x)| \leq k$, $|P| \leq k$. 
Therefore, there are at least $8k-k =7k$ possible values that $\chi_c$ can take  
such that $S_{xy}$ is favorably colored in $\mathcal{C}_i^{v_y = \chi_c}$. 
For $T_{yz}$ also to be 
favorably colored in $\mathcal{C}_i^{v_y = \chi_c}$, the only thing required is 
that $\chi_c \neq 
\mathcal{C}_i^{v_y = \chi_c}(v_z)$, since $v_z \in Z(v_y,\mathcal{C}_i)$ and therefore 
$S_{yz}$ is already favorably colored. 
This implies that there are at least $7k-1$ possible 
values that $\chi_c$ can take such that both $S_{xy}$ and 
$T_{yz}$ are favorably colored in $\mathcal{C}_i^{v_y = \chi_c}$. In other words, there 
are at least $7k-1$ \emph{good} colors  for $<v_x,v_z>$. Thus for each element  
in $A$, there are at least $7k-1$ colors \emph{good} for it. For each color 
$\chi_j \in \{\chi_1, \ldots , \chi_{8k}\}$, 
let $S^j = \{<v_x,v_z> \in A~|~\chi_j\mbox{ is \emph{good} for 
} <v_x,v_z>\} = X(v_y,\mathcal{C}_i^{v_y = \chi_j}) 
\times Y(v_y,\mathcal{C}_i^{v_y = \chi_j})$. 
Since there are at least $(7k-1)$ colors \emph{good} for each element 
in $A$, $\Sigma_{j \in [8k]}|S^j| \geq (7k-1)|A|$. Then by pigeonhole principle, 
there exists a $c \in [8k]$ such that $|S^c| = 
 |X(v_y,\mathcal{C}_i^{v_y = \chi_c})| 
\cdot |Y(v_y,\mathcal{C}_i^{v_y = \chi_c})| \geq \frac{(7k-1)}{8k}|A| 
= \frac{7k-1}{8k}|BNN_i[v_y]|\cdot |Z(v_y,\mathcal{C}_i)| \geq 
\frac{3}{4}|BNN_i[v_y]|\cdot |Z(v_y,\mathcal{C}_i)|$ elements of $A$. In other words, 
$|X(v_y,\mathcal{C}_i^{v_y = \chi_c})| \geq \frac{3}{4}|BNN_i[v_y]|$ and 
$|Y(v_y,\mathcal{C}_i^{v_y = \chi_c})| \geq \frac{3}{4}|Z(v_y,\mathcal{C}_i^{v_y = \chi_c})|$. 
\hfill \qed
\end{proof}

\begin{lemma}
\label{mbarlemma}
Let $\overline{m}_i = \Sigma_{y \in [n]}|FNN_i[v_y]|$. Then $\overline{m}_{i+1} \leq \frac{7}{16}\overline{m}_i$.
\end{lemma}
\begin{proof}
From Step $8$ of CONSTRUCT\_CUB\_REP($G$), we have 
$|FNN_{i+1}[v_y]| = \\ |FNN_i[v_y]| - |Y(v_y,\mathcal{C}_i)| \leq 
|FNN_i[v_y]| - \frac{3}{4}|Z(v_y,\mathcal{C}_i)|$ (using Lemma \ref{ratiolemma}). 
Taking summation over all $y \in [n]$, we get 
$\overline{m}_{i+1} \leq \overline{m}_{i} -
\frac{3}{4}\Sigma_{y \in [n]}|Z(v_y,\mathcal{C}_i)| = 
\overline{m}_{i} - 
\frac{3}{4}\Sigma_{y \in [n]}|X(v_y,\mathcal{C}_i)|$.  
The last equality comes from the fact that both 
$\Sigma_{y \in [n]}$ $|X(v_y,\mathcal{C}_i)|$ and 
$\Sigma_{y \in [n]}|Z(v_y,\mathcal{C}_i)|$ 
represent the number of HOPEFUL non-edges in $G$ with respect to 
colorings $\mathcal{C}_1, \ldots , \mathcal{C}_i$. 
From Lemma \ref{ratiolemma}, we have $|X(v_y,\mathcal{C}_i)| 
\geq \frac{3}{4}|BNN_i[v_y]|$. Therefore, 
$\overline{m}_{i+1} \leq 
\overline{m}_{i} - 
(\frac{3}{4})^2\Sigma_{y \in [n]}$ $|BNN_i[v_y]|$.  
Since $\Sigma_{y \in [n]}$ $|BNN_i[v_y]| = \Sigma_{y \in [n]}|FNN_i[v_y]|$, we get 
$\overline{m}_{i+1} \leq \overline{m}_{i} - (\frac{3}{4})^2\Sigma_{y \in [n]}$ $|FNN_i[v_y]|$  
$=  \overline{m}_{i} - \frac{9}{16}\overline{m}_{i} 
= \frac{7}{16}\overline{m}_i$.
\hfill \qed
\end{proof}

\begin{theorem}
\label{AlgoTheorem}
Let $G$ be a $k$-degenerate graph. Algorithm CONSTRUCT\_CUB\_\\REP($G$) constructs a 
valid $8k(\lceil 2.42 \log n\rceil + 1)$ dimensional cube representation for $G$. 
\end{theorem}
\begin{proof}
 The algorithm constructs $\alpha$ colorings $\mathcal{C}_1, 
\mathcal{C}_2, \ldots , \mathcal{C}_{\alpha}$ of $V(G)$, where each 
coloring uses colors from the set $\{\chi_1, \ldots ,\chi_{8k}\}$. 
From Lemma \ref{mbarlemma}, we have $\overline{m}_{i+1} \leq \frac{7}{16}\overline{m}_i$. 
Also, $\overline{m}_1 = |\Sigma_{y \in [n]} FNN_1[v_y]| \leq n^2$. 
Putting $\alpha = (\lceil 2.42 \log n\rceil + 1) $, we get $\overline{m}_{\alpha} \leq 1$. 
That is, for every $y \in [n]$, $FNN_{\alpha + 1}[v_y] = EMPTY$. This means that, 
for every $(v_x,v_y) \notin E(G)$, where $v_x <_{\mathcal{D}} v_y$, there exists an $i \in [\alpha]$ 
such that $T_{xy}$ is favorably colored in $\mathcal{C}_i$. 
% Then Lemma \ref{coloringAndCubiLemma} 
% describes a method to construct $\alpha8k$ unit interval graphs $I_{i,j}$, where 
% $i \in \alpha, j \in [8k]$ such that $G = \bigcap_{i=1}^{\alpha} \bigcap_{j=1}^{8k} I_{i,j}$.  
Then by Lemma \ref{coloringAndCubiLemma} , $\cubi(G) \leq 8k(\lceil 2.42 \log n\rceil + 1)$. 
The procedure CONSTRUCT\_UNIT\_INTERVAL\_GRAPHS  
constructs $8k(\lceil 2.42 \log n\rceil + 1)$ unit interval graphs whose 
intersection gives $G$, as described in Lemma \ref{coloringAndCubiLemma}. 
Thus we prove the theorem. 
\hfill \bbox
\end{proof}

% \newpage
% \thispagestyle{empty}
% \newpage
% 

\begin{algorithm}
\label{construct_coloring_detailed}
/*All data structures are assumed to be global. \\
\textbf{Notational Note:} \\
Let $\mathcal{C}_i^{v_z}$ denote the partial coloring at the stage 
when we have colored the vertices $v_n$ to $v_z$. 
Let $\mathcal{C}_i^{v_z = \chi_c}$ denote the partial coloring that 
results if we extend $\mathcal{C}_i^{v_{z+1}}$ by assigning color $\chi_c$ to 
$v_z$. */
\begin{algorithmic}
\caption{CONSTRUCT\_COLORING($i$) /* detailed */}
\setcounter{line}{0}
\MYSTATE{Initialize $FNC[w][j] \leftarrow 0, \forall w \in [n], j \in [8k]$}
\MYSTATE{Initialize $HOPE\_MATRIX[w][z] \leftarrow 0, \forall w,z \in [n]$}
\MYSTATE{Initialize $DONE\_MATRIX[w][z] \leftarrow 0, \forall w,z \in [n]$}
\FOR{$y=n$ to $1$}
\FOR{each $\chi_c \in \{\chi_1, \ldots ,\chi_{8k}$ \}}
% \MYSTATE{Let $\mathcal{C}_i^{v_y = \chi_c}$ denote the partial coloring that 
% results if we extend $\mathcal{C}_i^{v_{y+1}}$ by assigning color $\chi_c$ to 
% $v_y$.}
\MYSTATE{Compute $X(v_y,\mathcal{C}_i^{v_y = \chi_c})$ /*as described in steps (a) and (b) below */\\
\hspace{.2in} (a) Initialize $X(v_y,\mathcal{C}_i^{v_y = \chi_c}) \leftarrow \emptyset$ \\
\hspace{.2in} (b) $\forall v_x \in BNN_i[v_y]$, if $FNC[x][c] = 0$, then \\
\hspace{.2in} SET $X(v_y,\mathcal{C}_i^{v_y = \chi_c}) \leftarrow 
X(v_y,\mathcal{C}_i^{v_y = \chi_c}) \cup \{v_x\}$}
 
\MYSTATE{Compute $Y(v_y,\mathcal{C}_i^{v_y = \chi_c})$ /*as described in steps (a) and (b) below */\\
\hspace{.2in} (a) Initialize $Y(v_y,\mathcal{C}_i^{v_y = \chi_c}) \leftarrow \emptyset$ \\
\hspace{.2in} (b) $\forall v_z \in FNN_i[v_y]$, if $\left( HOPE\_MATRIX[y][z]=1 \right)$ and  \\
\hspace{.2in} $\left( \mathcal{C}_i^{v_y = \chi_c}(v_z) \neq \chi_c \right)$, then 
 SET $Y(v_y,\mathcal{C}_i^{v_y = \chi_c}) \leftarrow Y(v_y,\mathcal{C}_i^{v_y = \chi_c}) 
\cup \{v_z\}$}
\MYSTATE{Compute $Z(v_y,\mathcal{C}_i^{v_y = \chi_c})$ /*as described in steps (a) and (b) below */\\
\hspace{.2in} (a) Initialize $Z(v_y,\mathcal{C}_i^{v_y = \chi_c}) \leftarrow \emptyset$ \\
\hspace{.2in} (b) $\forall v_z \in FNN_i[v_y]$, if $ HOPE\_MATRIX[y][z]=1 $, \\
\hspace{.2in} then SET $Z(v_y,\mathcal{C}_i^{v_y = \chi_c}) \leftarrow Z(v_y,\mathcal{C}_i^{v_y = \chi_c}) 
\cup \{v_z\}$} 
\IF{$|X(v_y,\mathcal{C}_i^{v_y = \chi_c})| \geq \frac{3}{4}|BNN_i[v_y]|$ and 
$|Y(v_y,\mathcal{C}_i^{v_y = \chi_c})| \geq \frac{3}{4}|Z(v_y,\mathcal{C}_i^{v_y = \chi_c})|$} 
% \COMMENT{Such a $c$ always exists. See Lemma \ref{ratiolemma}}
\MYSTATE{SET $\mathcal{C}_i^{v_y} \leftarrow \mathcal{C}_i^{v_y= \chi_c}$ (i.e. 
SET $\mathcal{C}_i(v_y) \leftarrow \chi_c$).} 
\MYSTATE{SET $X(v_y,\mathcal{C}_i^{v_y}) \leftarrow  X(v_y,\mathcal{C}_i^{v_y= \chi_c})$}
\MYSTATE{SET $Y(v_y,\mathcal{C}_i^{v_y}) \leftarrow  Y(v_y,\mathcal{C}_i^{v_y= \chi_c})$}
\MYSTATE{SET $Z(v_y,\mathcal{C}_i^{v_y}) \leftarrow  Z(v_y,\mathcal{C}_i^{v_y= \chi_c})$}

\MYSTATE{Update $FNC$ matrix. /* as described in step (a) below */ \\
\hspace{.2in} (a) $\forall v_x \in N_G^b(v_y)$, SET $FNC[x][c] \leftarrow 1$  
}
\MYSTATE{Update $HOPE\_MATRIX$ /* as described in step (a) below */ \\
\hspace{.2in} (a) $\forall v_x \in X(v_y,\mathcal{C}_i^{v_y})$, SET 
$HOPE\_MATRIX[x][y] \leftarrow 1$
}
\MYSTATE{Update $DONE\_MATRIX$ /* as described in step (a) below */ \\
\hspace{.2in} (a) $\forall v_z \in Y(v_y,\mathcal{C}_i^{v_y})$, SET 
$DONE\_MATRIX[y][z] \leftarrow 1$
}
\MYSTATE{BREAK.}
\ENDIF
\ENDFOR
% \MYSTATE{$X(v_y)$}
\ENDFOR
\FOR{$y=1$ to $n$}
\MYSTATE{Compute $W(v_y,\mathcal{C}_i)$ /*as described in steps (a) and (b) below */\\
\hspace{.2in} (a) Initialize $W(v_y,\mathcal{C}_i) \leftarrow \emptyset$ \\
\hspace{.2in} (b) $\forall v_x \in BNN_i[v_y]$, if $DONE\_MATRIX[x][y] = 1$, then \\
\hspace{.2in}  SET $W(v_y,\mathcal{C}_i) \leftarrow 
W(v_y,\mathcal{C}_i) \cup \{v_x\}$
}
\MYSTATE{SET $Y(v_y,\mathcal{C}_i) \leftarrow Y(v_y,\mathcal{C}_i^{v_1})$}
\ENDFOR
\MYSTATE{Return $\mathcal{C}_i$. }
\end{algorithmic}
\end{algorithm}

\subsubsection{Running Time Analysis}
\begin{lemma}
\label{functionRunTimeLemma}
 The procedure CONSTRUCT\_COLORING($i$) can be implemented to run in 
$O(k\overline{m}_i + kn)$ time, where $\overline{m}_i = \Sigma_{y \in [n]}|FNN_i[v_y]|$. 
\end{lemma}
\begin{proof}
A detailed description of the procedure is given in Algorithm 4.4. 
To implement the procedure efficiently, we make use of an $(n \times 8k)$ 
$0$-$1$ matrix, hereafter called $FNC$ (Forward Neighbor Color), and two 
$(n\times n)$ $0$-$1$ matrices named $HOPE\_MATRIX$ and $DONE\_MATRIX$ 
respectively. At the beginning of the procedure each of these matrices have all 
entries set to 0. As the procedure progresses, we change some of the entries to 1 
in such a way that, \\
$\forall w \in [n], j \in [8k], FNC[w][j] = 1 \iff \exists v_z \in N_G^f(v_w) 
\mbox{ such that } v_z 
\mbox{ is already } \\
\mbox{colored by the procedure with color } 
\chi_j. \\
\forall w,z \in [n], v_w \in BNN_i[v_z], HOPE\_MATRIX[w][z] = 1  \iff  S_{wz} 
\mbox{ is already} \\
 \mbox{ favorably colored by the procedure.} \\
\forall w,z \in [n], v_w \in BNN_i[v_z], DONE\_MATRIX[w][z] = 1 \iff T_{wz} 
\mbox{ is already} \\
 \mbox{ favorably colored by the procedure.} 
$
 
In order for the above matrices to satisfy their respective properties, the only thing 
that needs to be done is to update these matrices at each stage of the procedure. 
Consider the stage at which the procedure is extending partial coloring 
$\mathcal{C}_i^{v_y + 1}$ to $\mathcal{C}_i^{v_y}$ by assigning color 
$\chi_c$ to $v_y$. At this stage, the matrices $FNC$, $HOPE\_MATRIX$ and $DONE\_MATRIX$ 
are updated as described in steps $11 (a)$, $12 (a)$ and $13 (a)$ respectively.  
Note that this can be done in $O(|BNN_i[v_y]| + |FNN_i[v_y]| + |N_G^b(v_y)|)$ time. 
Steps 4(a)-(b), 5(a)-(b) and 6(a)-(b) compute 
$X(v_y,\mathcal{C}_i^{v_y = \chi_c}), Y(v_y,\mathcal{C}_i^{v_y = \chi_c})$ and 
$Z(v_y,\mathcal{C}_i^{v_y = \chi_c})$ respectively in 
$O(|BNN_i[v_y]| + |FNN_i[v_y]|)$ time. Computing $W(v_y, \mathcal{C}_i)$ is done 
in step 15 (a)-(b) in $O(|BNN_i[v_y]|)$ time. 

Since steps 4 to 14, in the worst case, are run for each $v_y \in V(G), 
\chi_c \in \{\chi_1, \ldots , \chi_{8k}\}$, the procedure runs in  
$O(k(\Sigma_{y \in [n]}(|BNN_i[v_y]| + |FNN_i[v_y]|) + \Sigma_{y \in [n]}|N_G^b(v_y)|))$ time.  
We know that $\Sigma_{y \in [n]}(|BNN_i[v_y]| + |FNN_i[v_y]|) = 2 \overline{m}_i$ and 
$\Sigma_{y \in [n]}|N_G^b(v_y)| = m \leq kn$. Hence the Lemma.
\hfill \qed
\end{proof}

\begin{theorem}
 CONSTRUCT\_CUB\_REP($G$) runs in $O(n^2k)$ time. 
\end{theorem}
\begin{proof}
The algorithm invokes the function CONSTRUCT\_COLORING($i$) $\alpha$ \\ times 
to construct colorings $\mathcal{C}_1, \ldots , \mathcal{C}_{\alpha}$ of $V(G)$. 
By Lemma \ref{functionRunTimeLemma}, to construct these $\alpha$ colorings it requires 
$O(\Sigma_{i=1}^{\alpha} (\overline{m}_ik) + \alpha kn)$ time. From Lemma \ref{mbarlemma}, 
we get that  $\Sigma_{i=1}^{\alpha} (\overline{m}_i)$ is $O(\overline{m})$. Since 
$\alpha = (\lceil 2.42 \log n \rceil + 1)$, the running time of the while loop in CONSTRUCT\_CUB\_REP($G$) is 
$O(\overline{m}k + nk\log n )$. It is easy to see that the procedure 
CONSTRUCT\_UNIT\_INTERVAL\_GRAPHS() runs in $O(nk \log n)$ time. 
Since $ \overline{m} \leq n^2$, 
CONSTRUCT\_CUB\_REP($G$) runs in $O(n^2k)$ time. 
 \hfill \bbox
\end{proof}

\section{Boxicity, Cubicity, and Crossing Number}
\label{CrossingSection}
Crossing number of a graph $G$, denoted as $CR(G)$, is the minimum number of crossing pairs of edges, 
over all drawings of $G$ in the plane. 
A graph $G$ is planar if and only if $CR(G) = 0$. Determination of the crossing number is 
an NP-complete problem. 

The following theorem is due to Pach and T{\'o}th \cite{PachToth}
\begin{theorem}
\label{PachTheorem}
 For a graph $G$ with $n$ vertices and $m \geq 7.5n$ edges, 
$CR(G) \geq \frac{1}{33.75}\frac{m^3}{n^2}$, and this 
estimate is tight up to a constant factor. 
\end{theorem}

The following claims follow from the above theorem.
\begin{claim}
\label{crossingdavclaim}
 For a graph $G$ on $n$ vertices and $m$ edges, if $CR(G) \leq t$, then 
$d_{av}(G) \leq 2 (\frac{33.75t}{n})^{1/3} + 15$. 
\end{claim}
\begin{proof}
If $m < 7.5n$, then $d_{av}(G) < 15$. Otherwise, we have 
$m \leq (33.75 n^2 t)^{1/3}$ implying that 
$d_{av}(G) \leq 2 (\frac{33.75t}{n})^{1/3}$. 
\hfill \qed 
\end{proof}
\begin{claim}
\label{crossingdegclaim}
For a graph $G$ on $n$ vertices and $m$ edges, if $CR(G) = t$, then $G$ is $\left(6.5  t^{1/4} + 15\right)$-degenerate.  
\end{claim}
\begin{proof}
From the definition of crossing number we know that $CR(G) \leq {m \choose 2} \leq n^4$. Hence, $n \geq t^{1/4}$. Then by Claim \ref{crossingdavclaim}, $d_{av}(G) \leq 6.5  t^{1/4} + 15$. Thus $G$ is $\left(6.5  t^{1/4} + 15\right)$-degenerate. 
\hfill \qed  
\end{proof}

\begin{lemma}
\label{S_BC_Blemma}
Consider a graph $G$ whose vertices are partitioned into two parts namely $V_A$ and $V_B$. That is, $V(G) = V_A \uplus V_B$. 
% Let $n = |V(G)|$. 
% Let $S_B(G)$ denote the graph with $V(S_B(G)) = V(G)$ and 
% $E(S_B(G)) = E(G) \setminus \{(u,v)~|~u,v \in V_B\}$. In other words, $S_B(G)$ 
% is obtained from $G$ by making $V_B$ a stable set (or an independent set). Let $C_B(G)$ denote the graph with $V(C_B(G)) = V(G)$ and 
% $E(C_B(G)) = E(G) \cup \{(u,v)~|~u,v \in V_B\}$. That is, $C_B(G)$ 
% is obtained from $G$ by making $V_B$ a clique. 
% Let $G_B$ denote the subgraph of $G$ induced on $V_B$. 
% Analogously, we define $S_A(G), C_A(G)$, and $G_A$. 
Then,  $\boxi(C_B(G)) \leq 2 \boxi(S_B(G))$. 
\end{lemma}
\begin{proof}
Proof of this lemma is very similar to the proof of 
Lemma $3$ in \cite{RogSunSiv} and hence we only give a brief 
outline of it here. 
Assume $\boxi(S_B(G)) = r$. Then by Lemma 
\ref{robertslem}, there exist $r$ interval graphs $I_1, \ldots , I_r$ such that 
$S_B(G) = I_1 \cap I_2 \cap \cdots \cap I_r$. For each $i \in [r]$, let 
$f_i$ denote an interval representation of $I_i$. From these $r$ interval 
graphs we construct $2r$ interval graphs $I_1', \ldots ,I_r', 
I_1'', \ldots ,I_r''$ as outlined below. 
Let $f_i'$, $f_i''$ denote interval representations of $I_i'$ and $I_i''$ 
respectively, where $i\in [r]$. 
\begin{eqnarray*}
\mbox{Construction of } f_i': \\
\forall u\in V_A,~f_i'(u)& = & f_i(u). \\
\forall u\in V_B,~f_i'(u)& = & [\min_{v\in V_B}(l(f_i(v))),r(f_i(u))]. \\
\mbox{Construction of } f_i'': \\
\forall u\in A,~f_i''(u)& = & f_i(u). \\
\forall u\in B,~f_i''(u)& = & [l(f_i(u)), \max_{v\in V_B}(r(f_i(v)))]. \\
\end{eqnarray*}
We leave it to the reader to verify that $C_B(G) = \bigcap_{i=1}^{r}(I_i' \cap I_i'')$. 
\hfill \qed 
\end{proof}

\begin{lemma}
\label{usefullemma}
Consider a graph $G$. Let vertices of $G$ be partitioned into two parts namely $V_A$ and $V_B$. That is, $V(G) = V_A \uplus V_B$. 
% Let $S_B(G)$ and $C_B(G)$ be graphs as defined in Lemma \ref{S_BC_Blemma}. Let $G_B$ denote the subgraph 
% of $G$ induced on $V_B$.  
Then, 
$\boxi(G) \leq 2\boxi(S_B(G)) + \boxi(G_B)$. 
%, \\ 
% (ii) $\cubi(H) \leq 2\boxi(S_B(H))\log n + \cubi(H_B)$.  
\end{lemma}
\begin{proof}
% Let $C_B(H)$ be the graph with $V(C_B(H)) = V(H)$ and $E(C_B(H)) = E(H) \cup 
% \{(u,v)~|~u,v \in V_B\}$. In other words, $C_B(H)$ is obtained from $H$ by 
% making $V_B$ a clique. 
Let $G'$ be the graph with $V(G') = V(G)$ and 
$E(G') = E(G) \cup \{(u,v)~|~u \in V_A, v \in V(G')\}$. That is, each $u \in V_A$ is made a universal vertex in $G'$. Observe that  
$G = C_B(G) \cap G'.$ Then by Lemma \ref{robertslem}, we have $\boxi(G) \leq \boxi(C_B(G)) + \boxi(G')$. Applying Lemma \ref{S_BC_Blemma}, we get 
\begin{eqnarray}
\label{boxofGineq}
\boxi(G) & \leq & 2\boxi(S_B(G)) + \boxi(G')
% \\
% \label{cubofHineq}
% \cubi(H) & \leq & \cubi(C_B(H)) + \cubi(H')
\end{eqnarray}
% To prove the above inequality, we state two claims below:
\begin{claim}
\label{claim2}
 $\boxi(G') \leq \boxi(G_B)$. 
\end{claim}
Clearly, $G'$ is obtained from $G_B$ by adding universal vertices one after the other. 
Since adding a universal vertex to a graph does not increase its boxicity, 
$\boxi(G') \leq \boxi(G_B)$. 
% Let  $H_B^i$ be the graph formed by adding
% a vertex $w_i$ to $H_B^{i-1}$ such that $V(H_B^i)=V(H_B^{i-1})\cup \{w_i\}$ and $E(H_B^i)=E(H_B^{i-1})\cup
% \{(v,w_i)~|~v\in V(H_B^{i-1})\}$. That is, $w_i$ is a universal vertex in $H_B^i$. Let $H_B^0 \equiv H_B$. 
% Since adding a universal vertex to a graph does not increase its boxicity, $\boxi(H_B^i) \leq \boxi(H_B^{i-1})$, 
% $\forall i > 0$. From the definitions of $H_B$ and $H'$, we infer that  
% $ H'\equiv H_B^j$, for some $j\in \mathbb{N}$. 
% Therefore, $\boxi(H') \leq \boxi(H_B)$. 

Combining Inequality \ref{boxofGineq} and Claim \ref{claim2}, we get
$\boxi(G) \leq 2\boxi(S_B(G)) + \boxi(G_B)$. 
\hfill \qed
\end{proof}

\subsection{Boxicity and Crossing Number}
\begin{theorem}
\label{boxCrossingThm}
 For a graph $G$ with $CR(G) = t$, $\boxi(G) \leq 
66\cdot t^{\frac{1}{4}}{\lceil\log 4t\rceil}^{\frac{3}{4}} + 6$. 
% 24^{\frac{3}{4}} t^{\frac{1}{4}} (\log 4t)^{\frac{3}{4}} + 
% \frac{24^{\frac{3}{4}}t^{\frac{1}{4}}}{(\log 4t)^{\frac{1}{4}}} +  
% \frac{24^{\frac{3}{4}} t^{\frac{1}{4}} (\log 4t)^{\frac{3}{4}}}{2} + 6.$. 
\end{theorem}
\begin{proof}
 Consider a drawing $P$ of $G$ with $t$ crossings. We say a vertex 
$v$ \emph{participates} in a given crossing in $P$, if at least one of the 
edges of the given crossing is incident on $v$. 

Partition the vertices of $G$ into two parts, namely $V_A$ and $V_B$, such that 
$V_B = \{v \in V(G)~|~v \mbox{ participates in some crossing in $P$} \}$ and 
$V_A = V(G) \setminus V_B$. 
% Let $S_B(G)$ be the graph with $V(S_B(G)) = V(G)$ and 
% $E(S_B(G)) = E(G) \setminus \{(u,v)~|~u,v \in V_B\}$. In other words, $S_B(G)$ 
% is obtained from $G$ by making $V_B$ a stable set. Let $G_B$ be the subgraph 
% of $G$ induced on $V_B$. 
Then by Lemma \ref{usefullemma}, 
\[\boxi(G) \leq 2\boxi(S_B(G)) + \boxi(G_B).\] 
Observe that $S_B(G)$ is a planar 
graph and hence its boxicity is at most 3 (see \cite{Thoma1}). 
Therefore, $\boxi(G) \leq 6 + \boxi(G_B)$.  
For ease of notation, let $H \equiv G_B$. Then, 
\begin{eqnarray}
\label{crossingineq}
\boxi(G) \leq 6 + \boxi(H). 
\end{eqnarray}
We have $CR(H) = CR(G) = t$. Let $n = |V(H)|$ and 
$m=|E(H)|$. At most 4 vertices participate in a given crossing. 
Since each vertex in $H$ participates 
in some crossing in $P$, we get 
\[n \leq 4t.\] 

Let $V(H) = \{v_1, \ldots , v_n\}$. 
Let $v_1, \ldots , v_n$ be an ordering of the vertices of $H$, such that 
for each $i \in [n]$, $d_{H_i}(v_i) \leq d_{H_i}(v), \forall v \in V(H_i)$, 
where $H_i$ denotes the subgraph of $H$ induced on vertex set 
$\{v_i, \ldots , v_n\}$.
Let $k = \left(\frac{33.75}{3}\right)^{\frac{1}{4}}\left(\frac{t}{\lceil \log 4t \rceil}\right)^{\frac{1}{4}}$. Let 
$x = \min (\{i \in [n]~|~d_{H_i}(v_i) > k\})$. 
Partition $V(H)$ into two parts, namely $V_C = \{v_1, \ldots, v_{x-1}\}$ and 
$V_D = \{v_x, \ldots , v_n\}$. 
% Let $S_D(H)$ be the graph with $V(S_D(H)) = V(H)$ and 
% $E(S_D(H)) = E(H) \setminus \{(u,v)~|~u,v \in V_D\}$. In other words, $S_D(H)$ 
% is obtained from $H$ by making $V_D$ a stable set. Let $H_D$ be the subgraph 
% of $H$ induced on $V_D$. 
Then by Lemma \ref{usefullemma}, 
\[\boxi(H) \leq 2\boxi(S_D(H)) + \boxi(H_D).\]
Note that $S_D(H)$ is 
$k$-degenerate. If $k=1$, then $S_D(H)$ is a forest and hence its boxicity is at most 2. 
Suppose $k>1$. Then by Theorem \ref{probabilisticTheorem}, 
$\boxi(S_D(H)) \leq \cubi(S_D(H)) \leq (k+2)\lceil2e \log n \rceil \leq 12k \lceil \log (4t) \rceil \leq 
12 \left(\frac{33.75}{3}\right)^{\frac{1}{4}} t^{\frac{1}{4}} {\lceil \log 4t \rceil}^{\frac{3}{4}}$. 
Thus we have, 
\begin{eqnarray}
\label{ineq1}
 \boxi(H) \leq 24 \left(\frac{33.75}{3}\right)^{\frac{1}{4}} t^{\frac{1}{4}} {\lceil \log 4t \rceil}^{\frac{3}{4}}
 + \boxi(H_D).
\end{eqnarray}
Since $H_D \equiv H_x$, $v_x$ is a minimum degree vertex of $H_D$. 
% $d_{H_D}(v) \geq d_{H_D}(v_x) > k$. 
Therefore, 
$d_{av}(H_D) > d_{H_D}(v_x) > k$. Then by Claim \ref{crossingdavclaim}, we have 
\[k=\left(\frac{33.75}{3}\right)^{\frac{1}{4}}\left(\frac{t}{\lceil \log 4t 
\rceil}\right)^{\frac{1}{4}} < d_{av}(H_D) \leq 2 \left(\frac{33.75t}{|V(H_D)|}\right)^{1/3} + 15.\]
From this, we get $|V(H_D)| \leq 48^{\frac{3}{4}} (33.75t)^{\frac{1}{4}} {\lceil \log 4t \rceil}^{\frac{3}{4}}$. 
Since boxicity of a graph is at most half the number of its vertices\cite{Roberts} , we get 
$\boxi(H_D) \leq \frac{48^{\frac{3}{4}} (33.75t)^{\frac{1}{4}} {\lceil \log 4t \rceil}^{\frac{3}{4}}}{2}$. 
Substituting this in Inequality \ref{ineq1}, we get 
\[\boxi(H) \leq 66 t^{\frac{1}{4}}{\lceil\log 4t\rceil}^{\frac{3}{4}}
\]
Therefore from Inequality \ref{crossingineq} ,we get 
\[\boxi(G) \leq 66t^{\frac{1}{4}}{\lceil\log 4t\rceil}^{\frac{3}{4}} + 6.\]

% From Lemma \ref{crossingdavclaim}, we can say that $H$ is a $k$-degenerate graph, 
% where $k = 2 (\frac{33.75t}{n})^{1/3} + 15$. If $n \geq t^{1/4}(\log t)^{3/4}$, then 
% $k \leq \frac{t^{1/4}}{(\log t)^{1/4}}$. In this case, using Theorem \ref{}, we get 
% $\boxi(H) \leq \frac{t^{1/4}}{(\log t)^{1/4}} \log n$. Since $n \leq 4t$, we have 
% $\boxi(H) \leq t^{1/4}(\log t)^{3/4}$. Otherwise, if $n < t^{1/4}(\log t)^{3/4}$, 
% then $\boxi(H) \leq \frac{(t^{1/4}(\log t)^{3/4})}{2}$. Thus 
% $\boxi(H) \leq t^{1/4}(\log t)^{3/4}$. Substituting this in Inequality \ref{crossingineq}, 
% we get $\boxi(G) \leq 6 +  t^{1/4}(\log t)^{3/4}$. 
\hfill \bbox
\end{proof}
\subsubsection{Tightness of Theorem \ref{boxCrossingThm}: }
% Let $G \equiv K_{2,2,\ldots,2}$ denote the complete $\frac{n}{2}$-partite graph with 
% 2 vertices in each part and let $t = CR(G)$. From \cite{Roberts}, we know that 
% $\boxi(G) = \lfloor \frac{n}{2} \rfloor$. Since it is known that for any graph on $n$ vertices 
% its crossing number is upper bounded by $n^4$, we have $t \leq n^4$.  
% % From Theorem \ref{PachTheorem}, we get $t \geq  \frac{(n-2)^3n}{270}$.  
% % Since $\boxi(G) = \lfloor \frac{n}{2} \rfloor$, 
% Therefore, the bound given by Theorem \ref{boxCrossingThm} is tight up to a factor 
% of $O((\log t)^{\frac{3}{4}})$. 
Let $\mathbb{P}(n,p)$ be the probability space of height-2 posets with $n$ minimal elements forming set $A$ and $n$ maximal elements forming set $B$, where for any $a \in A$ and $b \in B$, $Pr[a<b] = p$. 
% The following theorem is due to Erd\H{o}s, Kierstead, and Trotter \cite{erdosKierstead}. 
% \begin{theorem}
%  For every $\epsilon > 0$, there exists $\delta > 0$ so that if $\frac{\log^{1+\epsilon}n }{n} < p < 1 - n^{\epsilon - 1}$, then, $dim(\mathcal{P}) > \frac{\delta pn \log(pn)}{1+\delta p \log(pn)}$ for almost all $\mathcal{P} \in \mathbb{P}(n,p)$. 
% \end{theorem}
Erd\H{o}s, Kierstead, and Trotter in \cite{erdosKierstead} proved that 
when $p=\frac{1}{\log n}$, for almost all posets $\mathcal{P} \in \mathbb{P}(n,p)$, $\Delta(G_{\mathcal{P}}) < \frac{\delta_1 n}{\log n}$ and $dim(\mathcal{P}) > \delta_2 n$, where $\delta_1$ and $\delta_2$ are some positive constants. Then by Theorem $1$ in \cite{DiptAdiga}, $\boxi(G_{\mathcal{P}}) \geq  \frac{dim(\mathcal{P})}{2} \geq \frac{\delta_2 n}{2}$. 

Let $t = CR(G_{\mathcal{P}})$ and let $m$ denote the number of edges of $G_{\mathcal{P}}$. It follows from definition of crossing number that $t \leq {m \choose 2} \leq m^2 \leq (n\Delta(G_{\mathcal{P}}))^2 \leq (\frac{\delta_1 n^2}{\log n})^2 \leq \frac{\delta_1^2 n^4}{(\log n)^2}$. Since $t \leq m^2 \leq n^4$, we have $n \geq t^{1/4}$ and thereby $\log n \geq \frac{1}{4}\log t$. Thus, $t \leq \frac{\delta_1^2 n^4}{(\frac{1}{4} \log t)^2} = \frac{16\delta_1^2 n^4}{(\log t)^2}$. From Theorem \ref{boxCrossingThm}, we have $\boxi(G_{\mathcal{P}}) \leq ct^{1/4}(\log t)^{3/4} + d \leq cn(\log t)^{1/4} + d$, where $c$ and $d$ are some constants. Therefore, the bound given by Theorem \ref{boxCrossingThm} is tight up to a factor of $O((\log t)^{\frac{1}{4}})$.

\subsection{Cubicity and Crossing Number}
\begin{theorem}
For a graph $G$ with $CR(G) = t$, $\cubi(G) \leq 6 \log_2n + \left(6.5  t^{1/4} + 17\right)$ $\lceil 2e\log (4t)\rceil$.  
\end{theorem}
\begin{proof}
Consider a drawing $P$ of $G$ with $t$ crossings. 
% We say a vertex 
% $v$ \emph{participates} in a given crossing in $P$, if at least one of the 
% edges of the given crossing is incident on $v$. 
As in Theorem \ref{boxCrossingThm}, partition the vertices of $G$ into two parts, namely $V_A$ and $V_B$, such that 
$V_B = \{v \in V(G)~|~v \mbox{ participates in some crossing in $P$} \}$ and 
$V_A = V(G) \setminus V_B$. 
% Let $S_B(G)$ be the graph with $V(S_B(G)) = V(G)$ and 
% $E(S_B(G)) = E(G) \setminus \{(u,v)~|~u,v \in V_B\}$. In other words, $S_B(G)$ 
% is obtained from $G$ by making $V_B$ a stable set. Let $G_B$ be the subgraph 
% of $G$ induced on $V_B$. 
Let $G'$ be the graph with $V(G') = V(G)$ and 
$E(G') = E(G) \cup \{(u,v)~|~u \in V_A, v \in V(G')\}$. That is, each $u \in V_A$ is made a universal vertex in $G'$. Observe that $G = C_B(G) \cap G'.$ 
Then by Lemma \ref{lemrobertscub}, 
$$ \cubi(G) \leq \cubi(C_B(G)) + \cubi(G')$$
It is shown in \cite{CubBox} that cubicity of a graph is at most $\log_2n$ times its boxicity. Applying this result, we get
\begin{eqnarray*}
 \cubi(G) & \leq & (\log_2n) \boxi(C_B(G)) + \cubi(G')  \\
& \leq & (2\log_2n) \boxi(S_B(G)) + \cubi(G') \mbox{ (by Lemma \ref{S_BC_Blemma})} 
\end{eqnarray*}
Observe that $S_B(G)$ is a planar 
graph and hence its boxicity is at most 3 (see \cite{Thoma1}). 
Therefore, 
\begin{eqnarray}
\label{cubCrossingIneq}
\cubi(G) \leq 6\log_2n + \cubi(G') 
\end{eqnarray}
Observe that $G'$ is the graph with $V(G') = V(G_B) \uplus V_A$ and $E(G') = E(G_B) \cup \{(u,v)~|~u \in V_A, v \in V(G')\}$. 
% can be constructed from $G_B$ by the following method: add each vertex of $V_A$ to $G_B$ one by one, where  the new vertex added at each step is a universal
%  vertex in the resultant graph.  
Since $CR(G_B) = CR(G) = t$, by Claim  \ref{crossingdegclaim}, $G_B$ is $\left(6.5  t^{1/4} + 15\right)$-degenerate. Then by Corollary \ref{ProbabilisticThmCor}, $\cubi(G') \leq \left(6.5  t^{1/4} + 17\right)$ $\lceil 2e\log (|V_B|)\rceil $. 
We know that at most 4 vertices participate in a given crossing. 
Since each vertex in $G_B$ participates in some crossing in $P$, we get $|V_B| \leq 4t.$ 
Thus,  $\cubi(G') \leq \left(6.5  t^{1/4} + 17\right) \lceil 2e\log (4t)\rceil $. Substituting for $\cubi(G')$ in Inequality (\ref{cubCrossingIneq}), we get 
$$ \cubi(G) \leq 6 \log_2n + \left(6.5  t^{1/4} + 17\right) \lceil 2e\log (4t)\rceil. $$
\hfill \bbox
\end{proof}

\section{Cubicity of Random Graphs}\label{CubicityRandGraph}
Given $n$ and $m$, in order to prove that almost all graphs in $\mathcal{G}(n,m)$ model have cubicity $O(\frac{2m}{n}\log n)$, we first show that cubicity of almost all graphs in $\mathcal{G}(n,p)$ model, where $p = \left(\frac{2m}{n}\right)\frac{1}{n-1} = \frac{m}{{n \choose 2}}$,  is $O(\frac{2m}{n}\log n)$. 
% We assume $c \geq 1$ as we are only interested in connected graphs. In order to prove the above result we first show that cubicity of almost all graphs in $\mathcal{G}(n,p)$ model is $O(c \log n)$, where $p = \frac{c}{n-1} = \frac{2m}{n(n-1)} = \frac{m}{{n \choose 2}}$. 
We then use a result in \cite{bollobas} to convert the result for graphs in $\mathcal{G}(n,p)$ model to  those in $\mathcal{G}(n,m)$ model. To show that almost all graphs in $\mathcal{G}(n,p)$ model have cubicity $O(\frac{2m}{n} \log n)$, we prove the following lemma. Then by Theorem \ref{probabilisticTheorem}, the desired result follows. 

\begin{lemma}
\label{randDegeneracyLemma}
For a random graph $G \in \mathcal{G}(n,p)$, where $p = \frac{c}{n-1}$ and $1 \leq c \leq n-1$, $Pr[G \mbox{ is } 4ec \mbox{-degenerate} ] \geq 1 - \frac{1}{\Omega(n^2)}$.  
\end{lemma}
\begin{proof}
In order to show that a given graph $G$ is $k$-degenerate it is enough to show that every induced subgraph of $G$ has average degree at most $k$. That is, for every $H$ which is an induced subgraph of $G$, $|E(H)| \leq \frac{|V(H)|k}{2}$. Below we prove that for almost all graphs $G \in \mathcal{G}(n,p)$, every induced subgraph $H$ of $G$ has $E(H) < |V(H)|2ec$. 

We use the following version of the Chernoff bound (refer page 64 of \cite{mitzenmacher}) in our proof
$$ Pr[X \geq (1+\delta)\mu] < \left(\frac{e^\delta}{(1+\delta)^{(1+\delta)}} \right)^{\mu} \leq \frac{1}{2^{(1+\delta) \mu \log_2 (\frac{1+\delta}{e})}}$$, where $X$ is a summation of independent Bernoulli random variables, $\mu \geq E[X]$, and $\delta$ is any positive constant. 

Let $G \in \mathcal{G}(n,p)$ be a random graph, where $p = \frac{c}{n-1}$. Let $H$ be an induced subgraph of $G$ with $|V(H)| = n\alpha$, where $ 0 < \alpha \leq 1$. Let $Y_H$ be a random variable that represents the number of edges in $H$. For any $v \in V(H)$, let $d_H(v)$ denote the degree of $v$ in $H$. Then, $E[d_H(v)] = p(n\alpha -1) = \frac{c(n\alpha -1)}{n-1} \leq \frac{c(n\alpha -\alpha)}{n-1} \leq c\alpha$ and $E[Y_H] = E[\frac{1}{2}\Sigma_{v \in V(H)}d_H(v)] = \frac{1}{2} \Sigma_{v \in V(H)}E[d_H(v)] \leq \frac{n \alpha^2 c}{2}$. Let $\delta = \frac{4e}{\alpha} - 1$ and $\mu = \frac{n\alpha^2 c}{2} \geq E[Y_H]$. Applying Chernoff bound, we get $Pr[Y_H \geq 2en\alpha c] \leq \frac{1}{2^{2en\alpha c \log_2(\frac{4}{\alpha})}}$. 
%  =  \frac{1}{2^{2en\alpha c \log_2(\frac{4e}{e\alpha})}}$. 
Here we split the proof into two cases: \\
\textbf{case $\frac{1}{4} \leq \alpha \leq 1$:} Then, $Pr[Y_H \geq 2en\alpha c] \leq \frac{1}{2^{(1/2)enc \log_2(\frac{4}{1})}} = \frac{1}{2^{enc}}$. Since $c \geq 1$, we get $Pr[Y_H \geq 2en\alpha c] \leq \frac{1}{2^{en}}$. Applying union bound it follows that, 
$$ Pr[\bigcup_{H: |V(H)| \geq \frac{n}{4}} Y_H \geq 2en\alpha c] \leq \frac{1}{2^{en}}\sum_{\alpha = \frac{1}{4}}^1 {n \choose n\alpha} \leq \frac{2^n}{2^{en}} = \frac{1}{2^{(e-1)n}}.$$
\\
\textbf{case $0 < \alpha < \frac{1}{4}$:} Here we use the following expression given in page $17$ of \cite{jukna}
 while taking union bound: ${n \choose n\alpha} \leq 2^{nH(\alpha)}$, where $H(\alpha) = \alpha
 \log_2 (\frac{1}{\alpha}) + (1 - \alpha) \log_2 (\frac{1}{1-\alpha})$ is the binary entropy
 function. This inequality can be proved using the Stirling's formula for factorials. Since $\alpha < \frac{1}{4}$, we have $\alpha
 \log_2 (\frac{1}{\alpha}) > (1 - \alpha) \log_2 (\frac{1}{1-\alpha})$. Therefore, $H(\alpha) \leq 2\alpha \log_2 \frac{1}{\alpha}$.
 Hence, 
\begin{eqnarray*}
Pr[\bigcup_{1 \leq n\alpha \leq \frac{n}{4}} ~\bigcup_{H: |V(H)| =
 n\alpha}  (Y_H \geq 2en\alpha c)] & \leq & \frac{n}{4} {n \choose n\alpha} \frac{1}{2^{2en\alpha
 c \log_2(\frac{4}{\alpha})}} \\ 
& \leq &  \frac{2^{\log_2(\frac{n}{4})} 2^{2n\alpha \log_2(\frac{1}{\alpha})}}{2^{2en\alpha c
 \log_2(\frac{4}{\alpha})}} \\ 
& \leq & \frac{2^{2n\alpha \log_2(\frac{1}{\alpha}) + \log_2n}}{2^{4en\alpha c + 2en\alpha c 
 \log_2(\frac{1}{\alpha})}} \\ 
& =& \frac{1}{2^{4en\alpha c + (2ec - 2)n \alpha \log_2(\frac{1}{\alpha}) - \log_2 n}} 
\end{eqnarray*}
Since $c \geq 1$ and $\alpha \geq \frac{1}{n}$, we get 
\begin{eqnarray}
\label{probRandIneq}
Pr[\bigcup_{1 \leq n\alpha \leq \frac{n}{4}} ~ \bigcup_{H:|V(H)| = n\alpha} (Y_H \geq 2en\alpha c)] & \leq & \frac{1}{2^{4e + (2ec - 2)n \alpha \log_2(\frac{1}{\alpha}) - \log_2 n}}
\end{eqnarray}
It is easy to see that the function $f(\alpha) = \alpha \log_2(\frac{1}{\alpha})$ is an increasing function, when $\alpha < \frac{1}{4}$. We have $\frac{1}{n} \leq \alpha < \frac{1}{4}$. Hence $f(\alpha) \geq f(1/n) = \frac{\log_2n}{n}$, when $\alpha < \frac{1}{4}$. Applying this to Inequality \ref{probRandIneq}, we get 
\begin{eqnarray*}
Pr[\bigcup_{1 \leq n\alpha \leq \frac{n}{4}} ~ \bigcup_{H:|V(H)| = n\alpha} (Y_H \geq 2en\alpha c)] & \leq & \frac{1}{2^{4e + (2ec - 3)\log_2 n}}
\end{eqnarray*}

Thus we say that the probability of any subgraph of $G$ to have its average degree greater than $(4ec + 1)$ is at most $\frac{1}{\Omega(n^2)}$. In other words, $G$ is $4ec$-degenerate with probability at least  $1 - \frac{1}{\Omega(n^2)}$. 
\hfill \qed
\end{proof}
\begin{theorem}
\label{cubRandGNP}
For a random graph $G \in \mathcal{G}(n,p)$, where $ p = \frac{c}{n-1}$ and $1\leq c \leq n-1$,  $Pr[\cubi(G) \notin O(c \log n)] \leq \frac{1}{\Omega(n^2)}$.  
\end{theorem}
\begin{proof}
% Here $p = \frac{m}{{n \choose 2}} = \left(\frac{2m}{n}\right)\frac{1}{n-1}$. 
Proof follows directly from Theorem \ref{probabilisticTheorem} and Lemma \ref{randDegeneracyLemma}.  
\hfill \bbox
\end{proof}
It is shown in page 35 of \cite{bollobas} that , 
$$ P_m(Q) \leq 3\sqrt{m} P_p(Q)$$ 
where Q is a property of graphs of order $n$, and $P_m(Q)$ and $P_p(Q)$ are the probabilities
 of a graph chosen at random from the $\mathcal{G}(n,m)$ or the $\mathcal{G}(n,p)$ models
 respectively to have property $Q$ given that $p = \frac{m}{{n \choose 2}} = \left(\frac{2m}{n}\right)\frac{1}{n-1}$. Note that for any connected graph $G$ with at least $2$ vertices, $\frac{2m}{n} \geq \frac{2(n-1)}{n} \geq 1$. Since we are only interested in connected graphs, we assume $\frac{2m}{n} \geq 1$. Then by Theorem \ref{cubRandGNP}, for a random graph $G \in \mathcal{G}(n,p)$, where $ p = \left(\frac{2m}{n}\right)\frac{1}{n-1}$,  $Pr[\cubi(G) \notin O(\frac{2m}{n} \log n)] \leq \frac{1}{\Omega(n^2)}$. Combining this result with the result shown in \cite{bollobas}, we say that for a random graph $G \in \mathcal{G}(n,m)$, $Pr[\cubi(G) \notin O(\frac{2m}{n} \log n)] \leq \frac{3\sqrt{m}}{\Omega(n^2)} \leq \frac{1}{\Omega(n)}$. 
% As $c = \frac{2m}{n} = d_{av}$, we have  shown that cubicity of almost all graphs is $O(d_{av} \log n)$, where $d_{av}$ is the average degree of the graph under consideration. 
Thus we have the following theorem. 

\begin{theorem}
 For a random graph $G \in \mathcal{G}(n,m)$, $Pr[\cubi(G) \in O(\frac{2m}{n} \log n)] \geq 1 -
 \frac{1}{\Omega(n)}$. 
\end{theorem}

\bibliography{mathewref}
\end{document}